\documentclass{article}
\usepackage{amssymb,bm}
\usepackage{amsmath}
\usepackage{amsthm}
\usepackage{graphicx}
\usepackage[active]{srcltx} 
\usepackage{hyperref}
\hypersetup{pdfborder=0 0 0}

\newtheorem{theorem}{{\sc Theorem}}[section]

\newtheorem{lemma}[theorem]{{\sc Lemma}}
\newtheorem{corollary}[theorem]{Corollary}

\newtheorem{definition}[theorem]{Definition}

\newcommand{\bb}[1]{\mathbb{ #1}}

\newcommand{\Sym}{\mathrm{Sym}}

\newcommand{\Trc}{\mathrm{Tr}\,}

\newcommand{\tns}[1]{#1\otimes #1}
\newcommand{\hf}{\displaystyle\frac{1}{2}}
\newcommand{\nth}[1]{\displaystyle\frac{1}{#1}}

\newcommand{\dif}[2]{\displaystyle\frac{\partial #1}{\partial #2}}
\newcommand{\Grad}{\nabla}
\newcommand{\Div}{\nabla \cdot}

\renewcommand{\Hat}[1]{\widehat{#1}}
\newcommand{\Tld}[1]{\widetilde{#1}}

\newcommand{\av}[1]{\langle #1 \rangle}

\def\XXint#1#2#3{{\setbox0=\hbox{$#1{#2#3}{\int}$ }
\vcenter{\hbox{$#2#3$ }}\kern-.6\wd0}}

\newcommand\myatop[2]{\genfrac{}{}{0pt}{}{#1}{#2}}

\newcommand{\re}{\Re\mathfrak{e}}

\newcommand{\lims}{\mathop{\overline\lim}}
\newcommand{\limi}{\mathop{\underline\lim}}

\newcommand{\bc}{boundary condition}

\newcommand{\lhs}{left-hand side}

\newcommand{\WLOG}{without loss of generality}
\newcommand{\nbh}{neighborhood}

\newcommand{\Ga}{\alpha}

\newcommand{\Gd}{\delta}

\newcommand{\Gf}{\phi}

\newcommand{\Gk}{\kappa}

\newcommand{\Gl}{\lambda}

\newcommand{\Gth}{\theta}

\newcommand{\GL}{\Lambda}

\newcommand{\GO}{\Omega}

\bmdefine\BGa{\alpha}
\bmdefine\BGb{\beta}
\bmdefine\BGd{\delta}
\bmdefine\BGe{\epsilon}
\bmdefine\BGve{\varepsilon}
\bmdefine\BGf{\phi}
\bmdefine\BGvf{\varphi}
\bmdefine\BGg{\gamma}
\bmdefine\BGc{\chi}
\bmdefine\BGi{\iota}
\bmdefine\BGk{\kappa}
\bmdefine\BGl{\lambda}
\bmdefine\BGn{\eta}
\bmdefine\BGm{\mu}
\bmdefine\BGv{\nu}
\bmdefine\BGp{\pi}
\bmdefine\BGth{\theta}
\bmdefine\BGvth{\vartheta}
\bmdefine\BGr{\rho}
\bmdefine\BGvr{\varrho}
\bmdefine\BGs{\sigma}
\bmdefine\BGvs{\varsigma}
\bmdefine\BGt{\tau}
\bmdefine\BGj{\tau}
\bmdefine\BGu{\upsilon}
\bmdefine\BGo{\omega}
\bmdefine\BGx{\xi}
\bmdefine\BGy{\psi}
\bmdefine\BGz{\zeta}
\bmdefine\BGD{\Delta}
\bmdefine\BGF{\Phi}
\bmdefine\BGG{\Gamma}
\bmdefine\BGL{\Lambda}
\bmdefine\BGP{\Pi}
\bmdefine\BGT{\Theta}
\bmdefine\BGS{\Sigma}
\bmdefine\BGU{\Upsilon}
\bmdefine\BGO{\Omega}
\bmdefine\BGX{\Xi}
\bmdefine\BGY{\Psi}

\newcommand{\CA}{{\mathcal A}}
\newcommand{\CB}{{\mathcal B}}
\newcommand{\CC}{{\mathcal C}}

\newcommand{\CE}{{\mathcal E}}
\newcommand{\CF}{{\mathcal F}}

\newcommand{\CL}{{\mathcal L}}
\newcommand{\CM}{{\mathcal M}}

\bmdefine\BCA{{\mathcal A}}
\bmdefine\BCB{{\mathcal B}}
\bmdefine\BCC{{\mathcal C}}
\bmdefine\BCD{{\mathcal D}}
\bmdefine\BCE{{\mathcal E}}
\bmdefine\BCF{{\mathcal F}}
\bmdefine\BCG{{\mathcal G}}
\bmdefine\BCH{{\mathcal H}}
\bmdefine\BCI{{\mathcal I}}
\bmdefine\BCJ{{\mathcal J}}
\bmdefine\BCK{{\mathcal K}}
\bmdefine\BCL{{\mathcal L}}
\bmdefine\BCM{{\mathcal M}}
\bmdefine\BCN{{\mathcal N}}
\bmdefine\BCO{{\mathcal O}}
\bmdefine\BCP{{\mathcal P}}
\bmdefine\BCQ{{\mathcal Q}}
\bmdefine\BCR{{\mathcal R}}
\bmdefine\BCS{{\mathcal S}}
\bmdefine\BCT{{\mathcal T}}
\bmdefine\BCU{{\mathcal U}}
\bmdefine\BCV{{\mathcal V}}
\bmdefine\BCW{{\mathcal W}}
\bmdefine\BCX{{\mathcal X}}
\bmdefine\BCY{{\mathcal Y}}
\bmdefine\BCZ{{\mathcal Z}}

\bmdefine\Bzr{ 0}
\bmdefine\Ba{ a}
\bmdefine\Bb{ b}
\bmdefine\Bc{ c}
\bmdefine\Bd{ d}
\bmdefine\Be{ e}
\bmdefine\Bf{ f}
\bmdefine\Bg{ g}
\bmdefine\Bh{ h}
\bmdefine\Bi{ i}
\bmdefine\Bj{ j}
\bmdefine\Bk{ k}
\bmdefine\Bl{ l}
\bmdefine\Bm{ m}
\bmdefine\Bn{ n}
\bmdefine\Bo{ o}
\bmdefine\Bp{ p}
\bmdefine\Bq{ q}
\bmdefine\Br{ r}
\bmdefine\Bs{ s}
\bmdefine\Bt{ t}
\bmdefine\Bu{ u}
\bmdefine\Bv{ v}
\bmdefine\Bw{ w}
\bmdefine\Bx{ x}
\bmdefine\By{ y}
\bmdefine\Bz{ z}
\bmdefine\BA{ A}
\bmdefine\BB{ B}
\bmdefine\BC{ C}
\bmdefine\BD{ D}
\bmdefine\BE{ E}
\bmdefine\BF{ F}
\bmdefine\BG{ G}
\bmdefine\BH{ H}
\bmdefine\BI{ I}
\bmdefine\BJ{ J}
\bmdefine\BK{ K}
\bmdefine\BL{ L}
\bmdefine\BM{ M}
\bmdefine\BN{ N}
\bmdefine\BO{ O}
\bmdefine\BP{ P}
\bmdefine\BQ{ Q}
\bmdefine\BR{ R}
\bmdefine\BS{ S}
\bmdefine\BT{ T}
\bmdefine\BU{ U}
\bmdefine\BV{ V}
\bmdefine\BW{ W}
\bmdefine\BX{ X}
\bmdefine\BY{ Y}
\bmdefine\BZ{ Z}

\newcommand{\SFL}{\mathsf{L}}

\title{
Rigorous derivation of the formula for the buckling load in axially compressed circular cylindrical shells
}
\author{Yury Grabovsky\footnote{Temple University, yury@temple.edu, 215-204-1650} 
\and Davit Harutyunyan\footnote{University of Utah}}
\begin{document}
\maketitle
\begin{abstract}
  The goal of this paper is to apply the recently developed theory of buckling
  of arbitrary slender bodies to a tractable yet non-trivial example of
  buckling in axially compressed circular cylindrical shells, regarded as
  three-dimensional hyperelastic bodies. The theory is based on a
  mathematically rigorous asymptotic analysis of the second variation of 3D,
  fully nonlinear elastic energy, as the shell's slenderness parameter goes to
  zero. Our main results are a rigorous proof of the classical formula for
  buckling load and the explicit expressions for the relative amplitudes of
  displacement components in single Fourier harmonics buckling modes, whose
  wave numbers are described by Koiter's circle. This work is also a part of
  an effort to quantify the sensitivity of the buckling load of axially
  compressed cylindrical shells to imperfections of load and shape.
\end{abstract}

\section{Introduction}
The buckling of rods, shells and plates is traditionally described in
mechanics textbooks as an instability in the framework of nonlinear shell
theory obtained by semi-rigorous dimension reduction of three-dimensional
nonlinear elasticity. While these theories are effective in describing large
deformations of rods and shells (including buckling), their heuristic nature
obscures the source of the discrepancy between theoretical and experimental
results, as is the case for axially compressed circular cylindrical shells
\cite{zmc02}. At the same time, a rigorously derived theory of bending of
shells \cite{fjmm03} captures deformations in the vicinity of relatively
smooth isometries of the middle surface. Unfortunately, the isometries of the
straight circular cylinder are non-smooth \cite{yosh55}. Our approach,
originating in \cite{grtr07}, is capable of giving a mathematically rigorous
treatment of buckling of slender bodies and determining whether the tacit
assumptions of the classical derivation are the source of the discrepancy with
experiment. In this paper, we apply our theory and obtain a mathematically
rigorous proof of the classical formula for buckling load
\cite{lorenz11,timosh14}. This result justifies the generally accepted
assumption that the paradoxical behavior of cylindrical shells in buckling is
due to the high sensitivity of the buckling load to imperfections
\cite{almr63,tenn64,wms65}. This phenomenon is commonly explained by the
instability of equilibrium states in the vicinity of the buckling point on the
bifurcation diagram \cite{koit45,wms65,buha66}. However, the exact mechanisms
of imperfection sensitivity are not fully understood, nor is there a reliable
theory capable of predicting experimentally observed buckling loads
\cite{lcp00,zmc02,hlp06}. While a full bifurcation analysis is necessary to
understand the stability of equilibria near the critical point, our method's
singular focus on the stability of the trivial branch gives access to the
scaling behavior of key measures of structural stability in the thin shell
limit. We have argued in \cite{grha} that axially compressed circular
cylindrical shells are susceptible to scaling instability of the critical
load, whereby the scaling exponent, and not just its coefficient, can
be affected by imperfections. The new analytical tools developed in
\cite{grha14} give hope for a path towards quantification of imperfection
sensitivity.

Our approach is based on the observation that the pre-buckled state is
governed by equations of linear elasticity \cite{grtr07}. At the
critical load, the linear elastic stress reaches a level at which the trivial
branch becomes unstable within the framework of 3D hyperelasticity. The origin
of this instability is completely geometric: the frame-indifference of the
energy density function implies\footnote{The assumption that the trivial
  deformation is stress-free is also essential.} non-convexity in the
compressive strain region. Since buckling occurs at relatively small
compressive loads, the material's stress-strain response is locally
linear. This explains why all classical formulas for buckling loads of
various slender structures involve only linear elastic moduli and hold
regardless of the material response model.

The significance of our approach is two-fold. First, it provides a common
platform to study buckling of arbitrary slender bodies. Second, its
conclusions are mathematically rigorous and its underlying assumptions
explicitly specified. The goal of this paper is to demonstrate the power and
flexibility of our method on the non-trivial, yet analytically solvable
example of the axially compressed circular cylindrical shell. Our analysis is
powered by asymptotically sharp Korn-like inequalities \cite{horg95,naza08},
where instead of bounding the $L^{2}$ norm of the displacement gradient by the
$L^{2}$ norm of the strain tensor, we bound the $L^{2}$ norm of individual
components of the gradient by the $L^{2}$ norm of the strain tensor. These
inequalities have been derived in our companion paper \cite{grha14}. The
method of buckling equivalence \cite{grtr07} provides flexibility by
furnishing a systematic way of discarding asymptotically insignificant terms,
while simplifying the variational functionals that characterize buckling.

The paper is organized as follows. In Section~\ref{sec:genth}, we describe the loading and
 corresponding trivial branch of an axially compressed cylindrical shell
treated as 3-dimensional hyperelastic body. We define stability of the
trivial branch in terms of the second variation of energy. Next, we describe
our approach from \cite{grtr07} and recall all necessary technical results from
\cite{grha14,grha} for the sake of completeness. In Section~\ref{sec:perfect}, we give
the rigorous derivation of the classical buckling load and identify the
explicit form of buckling modes. Our two most delicate results are a rigorous
proof of the existence of a buckling mode that is a single Fourier harmonic and
the linearization of the dependence of this buckling mode on the radial
variable---the two assumptions that are commonly made in the classical derivation
of the critical load formula.

\section{Axially compressed cylindrical shell}
\setcounter{equation}{0}
\label{sec:genth}
In this section we will give a mathematical formulation of the problem of buckling of axially
compressed cylindrical shell.

\subsection{Boundary conditions and trivial branch }
\label{sub:trbr}
Consider the circular cylindrical shell given in cylindrical coordinates
$(r,\Gth,z)$ as follows:
\[
\CC_{h}=I_{h}\times\bb{T}\times[0,L],\qquad I_{h}=[1-h/2,1+h/2],
\]
where $\bb{T}$ is a 1-dimensional torus (circle) describing $2\pi$-periodicity in
$\Gth$. Here $h$ is the slenderness parameter, equal to the ratio of the shell
thickness to the radius. In this paper we consider the axial compression of the shell where
the Lipschitz deformation $\By:\CC_{h}\to \mathbb R^3$ satisfies the boundary
conditions, given in cylindrical coordinates by
\begin{equation}
  \label{bc}
  y_{\Gth}(r,\Gth,0)=y_{z}(r,\Gth,0)=y_{\Gth}(r,\Gth,L)=0,\qquad
  y_{z}(r,\Gth,L)=(1-\Gl)L.
\end{equation}
The loading is parametrized by the compressive strain $\Gl>0$ in
the axial direction. The trivial deformation $\By(\Bx)=\Bx$ satisfies the \bc
s for $\Gl=0$. By a \emph{stable deformation} we mean a Lipschitz function
$\By(\Bx;h,\Gl)$, satisfying \bc s (\ref{bc}) and being a weak local
minimizer\footnote{A deformation $\By$ is called a weak local
  minimizer, if it delivers the smallest value of the energy $\CE(\By)$ among
  all Lipschitz function satisfying \bc s (\ref{bc}) that are sufficiently close to
  $\By$ in the $W^{1,\infty}$ norm.} of the energy functional
\[
\CE(\By)=\int_{\CC_{h}}W(\Grad\By)d\Bx
\]
among all Lipschitz functions satisfying (\ref{bc}). The energy density
function $W(\BF)$ is assumed to be three times continuously differentiable in a
\nbh\ of $\BF=\BI$. The key (and universal)
properties of $W(\BF)$ are
\begin{itemize}
\item[(P1)] Absence of prestress: $W_{\BF}(\BI)=\Bzr$;
\item[(P2)] Frame indifference: $W(\BF\BR)=W(\BF)$ for every $\BR\in SO(3)$;
\item[(P3)] Local stability of the trivial deformation $\By(\Bx)=\Bx$:
  \begin{equation}
    \label{Lcoerc}
    \av{\SFL_{0}\BGx,\BGx}> \Ga_{\SFL_{0}}|\BGx|^{2},\quad\BGx\in\Sym(\bb{R}^{3}),
  \end{equation}
  where $\Sym(\bb{R}^{3})$ is the space of symmetric $3\times 3$ matrices, and
  $\SFL_{0}=W_{\BF\BF}(\BI)$ is the linearly elastic tensor of material
  properties. Here, and elsewhere we use the notation
$\av{\BA,\BB}=\Trc(\BA\BB^{T})$ for the Frobenius inner product on the space of
$3\times 3$ matrices.
\end{itemize}
While this is not needed for general theory,
in this paper we will also assume that $W(\BF)$ is isotropic:
\begin{equation}
  \label{Wiso}
  W(\BF\BR)=W(\BF)\text{ for every }\BR\in SO(3).
\end{equation}
This assumption is necessary to obtain an explicit formula for the critical load.

Our goal is to examine stability of the homogeneous trivial branch
$\By(\Bx;h,\Gl)$ given in cylindrical coordinates by
\begin{equation}
  \label{trbr}
y_{r}=(a(\Gl)+1)r,\qquad y_{\Gth}=0,\qquad y_{z}=(1-\Gl)z,
\end{equation}
where the function $a(\Gl)$ is determined by the natural \bc s
\begin{equation}
  \label{nbc}
  \begin{cases}
  \BP(\Grad\By)\Be_{r}=\Bzr,&r=1\pm\frac{h}{2},\\
  \BP(\Grad\By)\Be_{z}\cdot\Be_{r}=0,&z=0,L,
  \end{cases}
\end{equation}
since uniform deformations always satisfy equations of equilibrium. Here
$\BP(\BF)=W_{\BF}(\BF)$, the gradient of $W$ with respect to $\BF$, is the
Piola-Kirchhoff stress tensor. We observe that
\begin{lemma}
  \label{lem:trbr}
Assume that $W(\BF)$ is three times continuously differentiable in a
\nbh\ of $\BF=\BI$, satisfies properties (P1)--(P3) and is isotropic
(i.e. satisfies (\ref{Wiso})). Then there exists
a unique function $a(\Gl)$, of class  $C^{2}$ on a \nbh\ of 0, such that
$a(0)=0$ and the natural \bc s (\ref{nbc}) are satisfied
\end{lemma}
\begin{proof}
  By (P2) $W(\BF)=\hat{W}(\BF^{T}\BF)$. The function $\hat{W}(\BC)$ is three times
  continuously differentiable in a \nbh\ of $\BC=\BI$. Thus,
\[
\BP(\BF)=W_{\BF}(\BF)=2\BF\hat{W}_{\BC}(\BF^{T}\BF).
\]
The isotropy (\ref{Wiso}) implies that $\hat{W}(\BR\BC\BR^{T})=\hat{W}(\BC)$
for all $\BR\in SO(3)$. Differentiating this relation in $\BR$ at $\BR=\BI$ we
conclude that $\hat{W}_{\BC}(\BC)$ must commute with $\BC$. In particular,
this implies that the matrix $\hat{W}_{\BC}(\BC)$ must be diagonal, whenever $\BC$
is diagonal. We compute that in cylindrical coordinates
\[
\BF=\Grad\By=\left[
  \begin{array}{ccc}
    1+a & 0 &0\\
    0 & 1+a& 0\\
    0 & 0& 1-\Gl
  \end{array}
\right],\quad
\BC=\BF^{T}\BF=\left[
  \begin{array}{ccc}
    (1+a)^{2} & 0 &0\\
    0 & (1+a)^{2}& 0\\
    0 & 0& (1-\Gl)^{2}
  \end{array}
\right]
\]
Hence, $\BP(\BF)$ is diagonal, and conditions (\ref{nbc}) reduce to a
single scalar equation
\begin{equation}
  \label{rnbc}
  \hat{W}_{\BC}((1+a)^{2}(\tns{\Be_{r}}+\tns{\Be_{\Gth}})+(1-\Gl)^{2}\tns{\Be_{z}})\Be_{r}\cdot\Be_{r}=0,
\end{equation}
where the \lhs\ of (\ref{rnbc}) is a twice continuously differentiable
function of $(\Gl,a)$. Condition (P1) implies that $(\Gl,a)=(0,0)$ is a
solution. The conclusion of the lemma is guaranteed by the implicit function
theorem, whose non-degeneracy condition reduces to
\begin{equation}
  \label{ndgn}
\hf\SFL_{0}(\tns{\Be_{r}}+\tns{\Be_{\Gth}})\Be_{r}\cdot\Be_{r}\not=0.  
\end{equation}
By assumption, $\SFL_{0}$ is isotropic, and the non-degeneracy condition (\ref{ndgn})
becomes $\Gk+\mu/3\not=0$, which is satisfied due to (P3). Here $\Gk$ and
$\mu$ are the bulk and shear moduli, respectively.
\end{proof}
It is important, that as $h\to 0$, the trivial branch does not blow up. In
fact, in our case the trivial branch is independent of $h$.

The general theory of buckling \cite{grtr07} is designed to detect the first
instability of a trivial branch in a slender body $\GO_{h}$ that is
well-described by linear elasticity. Here is the formal definition from
\cite{grtr07,grha}.
\begin{definition}
  \label{def:trbr}
  We call the family of Lipschitz equilibria $\By(\Bx;h,\Gl)$ of $\CE(\By)$ a
  \textbf{linearly elastic trivial branch} if there exist $h_{0}>0$ and $\Gl_{0}>0$,
  so that for every $h\in[0,h_{0}]$ and $\Gl\in[0,\Gl_{0}]$
\begin{itemize}
\item[(i)] $\By(\Bx;h,0)=\Bx$
\item[(ii)] There exist a family of Lipschitz functions $\Bu^{h}(\Bx)$,
  independent of $\Gl$, such that
\begin{equation}
  \label{fundass}
  \|\Grad\By(\Bx;h,\Gl)-\BI-\Gl\Grad\Bu^{h}(\Bx)\|_{L^{\infty}(\GO_{h})}\le C\Gl^{2},
\end{equation}
\item[(iii)]
  \begin{equation}
    \label{reglbda}
    \|\dif{(\Grad\By)}{\Gl}(\Bx;h,\Gl)-\Grad\Bu^{h}(\Bx)\|_{L^{\infty}(\GO_{h})}\le C\Gl
  \end{equation}
\end{itemize}
where the constant $C$ is independent of $h$ and $\Gl$.
\end{definition}
We remark, that the leading order asymptotics $\Bu^{h}(\Bx)$ of the nonlinear
trivial branch is nothing else but a linear elastic displacement, that can be
found by solving the equations of linear elasticity
$\Div(\SFL_{0}e(\Bu^{h}))=\Bzr$, augmented by the appropriate \bc s. Here
$e(\Bu^{h})=\hf(\Grad\Bu^{h}+(\Grad\Bu^{h})^{T})$ is the linear elastic
strain. The linear elastic trivial branch $\Gl\Bu^{h}(\Bx)$ depends only on
the linear elastic moduli $\SFL_{0}$, unlike the model-dependent nonlinear
trivial branch $\By(\Bx;h,\Gl)$.

The fact that our trivial branch (\ref{trbr}) satisfies all conditions in
Definition~\ref{def:trbr} is easy to verify. Here
\[
\Bu^{h}(\Bx)=\left.\dif{\By(\Bx;h,\Gl)}{\Gl}\right|_{\Gl=0}=
a'(0)r\Be_{r}-z\Be_{z}=\nu r\Be_{r}-z\Be_{z}
\]
is independent of $h$. Here we computed that $a'(0)=\nu$ (Poisson's ratio) by
differentiating (\ref{rnbc}) in $\Gl$ at $\Gl=0$.

\subsection{Stability of the trivial branch}
\label{sub:stab}
We define critical strain $\Gl_{\rm crit}$ in terms of the second variation of
energy
\begin{equation}
  \label{secvar}
\Gd^{2}\CE(\BGf;h,\Gl)=\int_{\CC_{h}}(W_{\BF\BF}(\Grad\By(\Bx;h,\Gl))\Grad\BGf,\Grad\BGf)d\Bx,
\end{equation}
defined on the space of admissible variations
\[
V_{h}^{\circ}=\{\BGf\in W^{1,\infty}(\CC_{h};\mathbb R^3):
\phi_{\Gth}(r,\Gth,0)=\phi_{z}(r,\Gth,0)=\phi_{\Gth}(r,\Gth,L)=\phi_{z}(r,\Gth,L)=0\}.
\]
By density of $W^{1,\infty}(\CC_{h};\mathbb R^3)$ in $W^{1,2}(\CC_{h};\mathbb
R^3)$ we extend the space of admissible variations from
$V_{h}^{\circ}$ to its closure $V_{h}$ in $W^{1,2}$.
\begin{equation}
\label{Breather}
V_{h}=\{\BGf\in W^{1,2}(\CC_{h};\mathbb R^3):
\phi_{\Gth}(r,\Gth,0)=\phi_{z}(r,\Gth,0)=\phi_{\Gth}(r,\Gth,L)=\phi_{z}(r,\Gth,L)=0\}.
\end{equation}

The critical strain $\Gl_{\rm crit}$ can be defined as follows.
\begin{equation}
  \label{truebl}
  \Gl_{\rm crit}(h)=\inf\{\Gl>0:\Gd^{2}\CE(\BGf;h,\Gl)<0\text{ for some }\BGf\in
  V_{h} \}.
\end{equation}
While this definition is unambiguous, it is inconvenient, since the critical strain
strongly depends on the choice of the nonlinear energy density function.
Instead, we will focus only on the leading order asymptotics of the
critical strain, as $h\to 0$. The corresponding buckling mode, to be defined
below, will also be understood in an asymptotic sense.
\begin{definition}
  \label{def:asymbm}
We say that a function $\Gl(h)\to 0$, as $h\to 0$ is a buckling load if
\begin{equation}
  \label{asymbl}
  \lim_{h\to 0}\frac{\Gl(h)}{\Gl_{\rm crit}(h)}=1.
\end{equation}
A \textbf{buckling mode} is a family of variations $\BGf_{h}\in V_{h}\setminus\{0\}$, such that
\begin{equation}
  \label{trubm}
\lim_{h\to 0}\frac{\Gd^{2}\CE(\BGf_{h};h,\Gl_{\rm crit}(h))}
{\Gl_{\rm crit}(h)\dif{(\Gd^{2}\CE)}{\Gl}(\BGf_{h};h,\Gl_{\rm crit}(h))}=0.
\end{equation}
\end{definition}
Targeting only the leading order asymptotics allows us to determine
critical strain and buckling modes from a \emph{constitutively linearized}
second variation \cite{grtr07}:
\begin{equation}
  \label{ccsv}
\Gd^{2}\CE_{\rm cl}(\BGf;h,\Gl)=\int_{\CC_{h}}\{\av{\SFL_{0}e(\BGf),e(\BGf)}+
\Gl\av{\BGs_{h},\Grad\BGf^{T}\Grad\BGf}\}d\Bx, \qquad\BGf\in V_{h},
\end{equation}
 and $\BGs_{h}$ is the linear elastic stress
\begin{equation}
  \label{minusstress}
 \BGs_{h}(\Bx)=\SFL_{0}e(\Bu^{h}(\Bx)).
\end{equation}
Since the first term in (\ref{ccsv}) is always non-negative we define the set
\begin{equation}
  \label{Adef}
\CA_{h}=\left\{\BGf\in V_{h}:\av{\BGs_{h},\Grad\BGf^{T}\Grad\BGf}<0\right\}
\end{equation}
of potentially destabilizing variations. The constitutively linearized
critical load will then be determined by minimizing the Rayleigh quotient
\begin{equation}
  \label{Rq}
\mathfrak{R}(h,\BGf)=-\frac{\int_{\GO_{h}}\av{\SFL_{0}e(\BGf),e(\BGf)}d\Bx}
{\int_{\GO_{h}}\av{\BGs_{h},\Grad\BGf^{T}\Grad\BGf}d\Bx}.
\end{equation}
The functional $\mathfrak{R}(h,\BGf)$ expresses the relative strength of the
destabilizing compressive stress, measured by the functional
\begin{equation}
  \label{compres.measure}
  \mathfrak{C}_{h}(\BGf)=\int_{\GO_{h}}\av{\BGs_{h},\Grad\BGf^{T}\Grad\BGf}d\Bx
\end{equation}
and the reserve of structural stability measured by the functional
\begin{equation}
\label{stabil.measure}
\mathfrak{S}_{h}(\BGf)=\int_{\GO_{h}}\av{\SFL_{0}e(\BGf),e(\BGf)}d\Bx.
\end{equation}
\begin{definition}
  \label{def:Bload}
The \textbf{constitutively linearized buckling load} $\Gl_{\rm cl}(h)$ is defined by
\begin{equation}
  \label{clin}
\Gl_{\rm cl}(h)=\inf_{\BGf\in\CA_{h}}\mathfrak{R}(h,\BGf).
\end{equation}
  We say that the family of variations $\{\BGf_{h}\in\CA_{h}:h\in(0,h_{0})\}$
  is a \textbf{constitutively linearized buckling mode} if
  \begin{equation}
    \label{clbm}
\lim_{h\to 0}\frac{\mathfrak{R}(h,\BGf_{h})}{\Gl_{\rm cl}(h)}=1.
  \end{equation}
\end{definition}
In \cite{grtr07} we have defined a measure of ``slenderness'' of the body in
terms of the Korn constant
\begin{equation}
  \label{KK}
  K(V_{h})=\inf_{\BGf\in V_{h}}\frac{\|e(\BGf)\|^{2}_{L^{2}(\GO_{h})}}{\|\Grad\BGf\|^{2}_{L^{2}(\GO_{h})}}.
\end{equation}
It is obvious, that if $K(V_{h})$ stays uniformly positive, then so does the
constitutively linearized second variation $\Gd^{2}\CE_{\rm cl}(\BGf;h,\Gl(h))$
as a quadratic form on $V_{h}$, for any $\Gl(h)\to 0$, as $h\to 0$.
\begin{definition}
  \label{def:slender}
We say that the  body $\GO_{h}$ is \textbf{slender} if
\begin{equation}
  \label{slender}
  \lim_{h\to 0}K(V_{h})=0.
\end{equation}
\end{definition}
This notion of slenderness requires not only geometric slenderness of the
domain but also traction-dominated \bc s conveniently encoded in the subspace
$V_{h}$, satisfying $W_{0}^{1,2}(\GO_{h};\bb{R}^{3})\subset V_{h}\subset
W^{1,2}(\GO_{h};\bb{R}^{3})$.

We can now state sufficient conditions, established in \cite{grha}, under which
the constitutively linearized buckling load and buckling mode, defined in
(\ref{clin})--(\ref{clbm}), verify Definition~\ref{def:asymbm}.
\begin{theorem}
  \label{th:crit}
Suppose that the body is slender in the sense of
Definition~\ref{def:slender}. Assume that the constitutively linearized
critical load $\Gl_{\rm cl}(h)$, defined in (\ref{clin}) satisfies
$\Gl_{\rm cl}(h)>0$ for all sufficiently small $h$ and
\begin{equation}
  \label{sufcond}
  \lim_{h\to 0}\frac{\Gl_{\rm cl}(h)^{2}}{K(V_{h})}=0.
\end{equation}
Then $\Gl_{\rm cl}(h)$ is the buckling load and any constitutively linearized
buckling mode $\BGf_{h}$ is a buckling mode in the sense of Definition~\ref{def:asymbm}.
\end{theorem}

Now we will show that Theorem~\ref{th:crit} applies to the axially compressed
circular cylindrical shells. The asymptotics of the Korn constant $K(V_{h})$,
as $h\to 0$, was established in \cite{grha14}.
\begin{theorem}
  \label{th:KCas}
Let $V_{h}$ be given by (\ref{Breather}). Then, there exist positive constants
$c(L)<C(L)$, depending only on $L$, such that
\begin{equation}
  \label{KI}
 c(L)h^{3/2}\leq K(V_{h})\leq C(L)h^{3/2}.
\end{equation}
\end{theorem}
In order to establish (\ref{sufcond}) we need to estimate $\Gl_{\rm cl}(h)$.
For the trivial branch (\ref{trbr}) we compute
\begin{equation}
  \label{perfect.stress.tensor}
\BGs_{h}=-E\tns{\Be_{z}},
\end{equation}
where $E$ is the Young's modulus. Hence,
\begin{equation}
  \label{ex.compres.measure}
\mathfrak{C}_{h}(\BGf)=
-E(\|\Gf_{r,z}\|^{2}+\|\Gf_{z,z}\|^{2}+\|\Gf_{\Gth,z}\|^{2}).
\end{equation}
where from now on $\|\cdot\|$ will always denote the $L^{2}$-norm on $\CC_{h}$.
In order to estimate $\Gl_{\rm cl}(h)$ we need to prove Korn-like inequalities
for the gradient components, $\Gf_{r,z}$, $\Gf_{z,z}$, and
$\Gf_{\Gth,z}$. This was done in \cite{grha14}.
\begin{theorem}
\label{th:KI}
There exist a constant $C(L)>0$ depending only on $L$ such that for any $\BGf\in V_{h}$ one has,
 \begin{equation}
  \label{thetaz}
  \|\phi_{\Gth,z}\|^{2}\le\frac{C(L)}{\sqrt{h}}\|e(\BGf)\|^{2},
\end{equation}
\begin{equation}
  \label{rz}
   \|\phi_{r,z}\|^{2}\le\frac{C(L)}{h}\|e(\BGf)\|^{2}.
\end{equation}
Moreover, the powers of $h$ in the inequalities (\ref{KI})--(\ref{rz}) are optimal,
achieved \emph{simultaneously} by the ansatz
\begin{equation}
\begin{cases}
\label{ansatz0}
\phi^{h}_{r}(r,\Gth,z)=&-W_{,\eta\eta}\left(\frac{\Gth}{\sqrt[4]{h}},z\right)\\[2ex]
\phi^{h}_{\Gth}(r,\Gth,z)=&r\sqrt[4]{h}W_{,\eta}\left(\frac{\Gth}{\sqrt[4]{h}},z\right)+
\frac{r-1}{\sqrt[4]{h}}W_{,\eta\eta\eta}\left(\frac{\Gth}{\sqrt[4]{h}},z\right),\\[2ex]
\phi^{h}_{z}(r,\Gth,z)=&(r-1)W_{,\eta\eta z}\left(\frac{\Gth}{\sqrt[4]{h}},z\right)
-\sqrt{h}W_{,z}\left(\frac{\Gth}{\sqrt[4]{h}},z\right),
\end{cases}
\end{equation}
 where $W(\eta,z)$ can be any smooth compactly supported function on
 $(-1,1)\times(0,L)$, with the understanding that the above formulas hold on a
 single period $\Gth\in[0,2\pi]$, while the function $\BGf^{h}(r,\Gth,z)$ is
 $2\pi$-periodic in $\Gth$.
\end{theorem}
\begin{corollary}
  \label{cor:lambda}
\begin{equation}
 \label{perfect.lambda.hat.upper}
 ch\le\Gl_{\rm cl}(h)\leq Ch.
 \end{equation}
\end{corollary}
\begin{proof}
  This is an immediate consequence of Theorem~\ref{th:KI}. The lower bound
  follows from inequalities (\ref{Lcoerc}), (\ref{thetaz}) and (\ref{rz}) (and
  also an obvious inequality $\|\phi_{z,z}\|\le\|e(\BGf)\|$). The upper bound
  follows from using a test function (\ref{ansatz0}) in the constitutively
  linearized second variation.
\end{proof}
Inequalities (\ref{KI}) and (\ref{perfect.lambda.hat.upper}) imply that the
condition (\ref{sufcond}) in Theorem~\ref{th:crit} is satisfied for the
axially compressed circular cylindrical shell.

\subsection{Buckling equivalence}
\label{sub:be}
The problem of finding the asymptotic behavior of the critical strain
$\Gl_{\rm crit}$ and the corresponding buckling mode, as $h\to 0$ now reduces
to minimization of the Rayleigh quotient (\ref{Rq}), which is expressed
entirely in terms of linear elastic data.  Even though this already represents
a significant simplification of our problem, its explicit solution is still
technically difficult. However, the asymptotic flexibility of the notion of
buckling load and buckling mode permits us to replace
$\mathfrak{R}(h,\BGf_{h})$ with an equivalent, but simpler functional. The
notion of buckling equivalence was introduced in \cite{grtr07} and developed
further in \cite{grha}. Here we give the relevant definition and theorems for
the sake of completeness.
\begin{definition}
  \label{def:equiv}
  Assume that $J(h,\BGf)$ is a variational functional defined on
  $\CB_h\subset\CA_{h}$. We say that the pair $(\CB_h, J(h,\BGf))$
  \textbf{characterizes buckling} if the following three conditions are
  satisfied
\begin{enumerate}
\item[(a)] Characterization of the buckling load:
If
\[
\Gl(h)=\inf_{\BGf\in\CB_{h}}J(h,\BGf),
\]
then $\Gl(h)$ is a buckling load in the sense of Definition~\ref{def:asymbm}.
\item[(b)] Minimizing property of the buckling
  mode: If $\BGf_{h}\in\CB_{h}$ is a buckling mode in the sense of
  Definition~\ref{def:asymbm}, then
  \begin{equation}
    \label{bmode}
\lim_{h\to 0}\frac{J(h,\BGf_{h})}{\Gl(h)}=1.
  \end{equation}
\item[(c)] Characterization of the buckling mode:
If $\BGf_{h}\in\CB_{h}$ satisfies (\ref{bmode}) then it is a buckling mode.
\end{enumerate}
\end{definition}
\begin{definition}
  \label{def:Bequivalence}
  Two pairs $(\CB_h, J(h,\BGf))$ and $(\CB'_h, J'(h,\BGf))$ are called
  \textbf{buckling equivalent} if the pair $(\CB_h,J(h,\BGf))$ characterizes buckling if
  and only if $(\CB'_h,J'(h,\BGf))$ does.
\end{definition}
Of course this definition becomes meaningful only if the pairs
$(\CB_h,J(h,\BGf))$ and $(\CB'_h,J'(h,\BGf))$ are related.

The following lemma has been proved in \cite{grha}.

\begin{lemma}
\label{lem:pairB_hJ}
Suppose the pair $(\CB_{h},J(h,\BGf))$ characterizes buckling.  Let
$\CB'_{h}\subset\CB_{h}$ be such that it contains a buckling mode. Then the
pair $(\CB'_{h},J(h,\BGf))$ characterizes buckling\footnote{This lemma
  highlights the fact that Part (b) in Definition~\ref{def:equiv} is designed
  to capture \emph{the} buckling mode. We make no attempt to characterize an
  infinite set of geometrically distinct, yet energetically equivalent
  buckling modes that exist in our example.}.
\end{lemma}
The key tool for simplification of functionals characterizing
buckling is the following theorem, \cite{grha}.
\begin{theorem}[Buckling equivalence]
\label{th:Bequivalence}
Suppose that $\Gl(h)$ is a buckling load in the sense of Definition~\ref{def:asymbm}.
If either
\begin{equation}
\label{J1J2}
\lim_{h\to 0}\Gl(h)\sup_{\BGf\in\CB_{h}}\left|\frac{1}{J_1(h,\BGf)}-\frac{1}{J_2(h,\BGf)}\right|=0,
\end{equation}
or
\begin{equation}
\label{J2J1}
\lim_{h\to 0}\nth{\Gl(h)}\sup_{\BGf\in\CB_{h}}|J_1(h,\BGf)-J_2(h,\BGf)|=0,
\end{equation}
then the pairs $(\CB_h, J_1(h,\BGf))$ and $(\CB_h, J_2(h,\BGf))$ are
buckling equivalent in the sense of Definition~\ref{def:Bequivalence}.
\end{theorem}
As an application we will simplify the denominator
in the functional $\mathfrak{R}(h,\BGf)$, given by (\ref{Rq}).   
Theorem~\ref{th:KI} suggests that $\|\Gf_{r,z}\|^{2}$ can be
much larger than $\|\Gf_{z,z}\|^{2}$ and $\|\Gf_{\Gth,z}\|^{2}$. Hence, we
will prove that we can to replace $\mathfrak{C}_{h}(\BGf)$, given by
(\ref{ex.compres.measure}), with $-E\|\Gf_{r,z}\|^{2}$. Hence, we define a
simplified functional
\[
\mathfrak{R}_{1}(h,\BGf)=\frac{\int_{\CC_{h}}\av{\Hat{\SFL}_{0}e(\BGf),e(\BGf)}d\Bx}
{\int_{\CC_{h}}|\phi_{r,z}|^{2}d\Bx},\qquad \Hat{\SFL}_{0}=\frac{\SFL_{0}}{E}.
\]
\begin{lemma}
  \label{lem:K1}
The pair $(\CA_h, \mathfrak{R}_{1}(h,\BGf))$ characterizes buckling.
\end{lemma}
\begin{proof}
By Theorem~\ref{th:KI} we have
\[
\left|\frac{1}{\mathfrak{R}(h,\BGf)}-\frac{1}{\mathfrak{R}_1(h,\BGf)}\right|=
\frac{\|\phi_{\Gth,z}\|^{2}+\|\phi_{z,z}\|^{2}}{\int_{\CC_{h}}\av{\Hat{\SFL}_{0}e(\BGf),e(\BGf)}d\Bx}
\leq\frac{C}{\sqrt{h}}.
\]
for every $\BGf\in V_h$. Condition (\ref{J1J2}) now follows from
(\ref{perfect.lambda.hat.upper}). Thus, by Theorem~\ref{th:Bequivalence}, the
pair $(\CA_h, \mathfrak{R}_1(h,\BGf))$ characterizes buckling.
\end{proof}

\section{Rigorous derivation of the classical  formula for the buckling load}
\setcounter{equation}{0}
\label{sec:perfect}
In this section we prove the classical asymptotic formula for the critical
strain \cite{lorenz11,timosh14}
\begin{equation}
  \label{answer}
  \Gl_{\rm crit}(h)\sim\frac{h}{\sqrt{3(1-\nu^{2})}}.
\end{equation}

\subsection{Restriction to a single Fourier mode}
\label{sec:Buck.modes.Fourier}
The goal of this section is to show that even if we shrink the space of
admissible variations to the set of single Fourier modes in $(\Gth,z)$, we still retain
the ability to characterize buckling. The first step is to define Fourier
modes by constructing an appropriate $2L$-periodic extension of $\BGf$ in $z$
variable. Since, no continuous $2L$-periodic extension $\Tld{\BGf}$ has the property that
$e(\Tld{\BGf})(r,\Gth,-z)=\pm e(\BGf)(r,\Gth,z)$, we will have to navigate around various
sign changes in components of $e(\BGf)$. We can handle this difficulty if
$\SFL_{0}$ is isotropic, which we have already assumed. It is easy to check that there
are only two possibilities that work\footnote{Meaning that each component of
  $e(\BGf)$ and its trace either changes sign or remains unchanged.}: odd
extension for $\phi_{r}$, $\phi_{\Gth}$, even for $\phi_{z}$, and even for
$\phi_{r}$, $\phi_{\Gth}$, odd for $\phi_{z}$. Since, $\phi_{r}$ is
unconstrained at the boundary $z=0,L$, only the latter possibility is
available to us. Hence, we expand $\phi_{r}$ and $\phi_{\Gth}$ is the cosine
series in $z$, while $\phi_{z}$ is represented by the sine series:
\begin{equation}
  \label{Fourier}
  \begin{cases}
\displaystyle\phi_{r}(r,\Gth,z)=\sum_{n\in\bb{Z}}\sum_{m=0}^{\infty}\Hat{\phi}_{r}(r,m,n)e^{in\Gth}
\cos\left(\frac{\pi m z}{L}\right),\\[2ex]
\displaystyle\phi_{\Gth}(r,\Gth,z)=\sum_{n\in\bb{Z}}\sum_{m=0}^{\infty}\Hat{\phi}_{\Gth}(r,m,n)e^{in\Gth}
\cos\left(\frac{\pi m z}{L}\right),\\[2ex]
\displaystyle\phi_{z}(r,\Gth,z)=\sum_{n\in\bb{Z}}\sum_{m=0}^{\infty}\Hat{\phi}_{z}(r,m,n)e^{in\Gth}
\sin\left(\frac{\pi m z}{L}\right).
\end{cases}
\end{equation}
While functions in $V_{h}$ can be represented by the expansion
(\ref{Fourier}), single Fourier modes do not belong to $V_{h}$. Yet, the
convenience of working with such simple test functions outweighs this
unfortunate circumstance, and hence, we switch (for the duration of technical
calculations) to the space
\begin{equation}
\label{Breather.average}
\tilde V_{h}=\left\{\BGf\in W^{1,2}(\CC_{h};\mathbb R^3):
\phi_{z}(r,\Gth,0)=\phi_{z}(r,\Gth,L)=\int_0^L\phi_{\Gth}(r,\Gth,z)dz=0\
\forall (r,\Gth)\in I_{h}\times\bb{T}\right\}.
\end{equation}
We will come
back at the very end to the space $V_{h}$ to get the desired result for our
original \bc s.

The space $\tilde V_{h}$ appears in our companion paper \cite{grha14} as
$V_h^3$, where the inequalities (\ref{KI}), (\ref{thetaz}) and (\ref{rz}) have
been proved for it. As a consequence, the estimates
(\ref{perfect.lambda.hat.upper}) hold for
\begin{equation}
  \label{ldtld}
\Tld{\Gl}(h)=\inf_{\BGf\in\tilde \CA_{h}}\mathfrak{R}(h,\BGf).
\end{equation}
We conclude that the pair $(\tilde\CA_{h},\mathfrak{R}(h,\BGf))$ characterizes
buckling (for the new \bc s associated with the space $\tilde V_{h}$).  In that case the proof of Lemma~\ref{lem:K1} carries with no change
for the space $\tilde V_{h}$. Hence, the pair
$(\tilde\CA_{h},\mathfrak{R}_{1}(h,\BGf))$ characterizes buckling as well.

We now define the single Fourier mode spaces $\CF(m,n)$. For any complex-valued function
$\Bf(r)=(f_{r}(r),f_{\Gth}(r),f_{\Gth}(r))$ and any $m\ge 1$, $n\ge 0$
we define
\[
\BGF_{m,n}(\Bf)=\left(f_{r}(r)\cos\left(\dfrac{\pi mz}{L}\right),
f_{\Gth}(r)\cos\left(\dfrac{\pi mz}{L}\right),f_{z}(r)\sin\left(\dfrac{\pi mz}{L}\right)
\right)e^{in\Gth},
\]
and
\begin{equation}
  \label{Xmndef}
 \CF(m,n)=\{\re(\BGF_{m,n}(\Bf)):\Bf\in C^{1}(I_{h};\bb{C}^{3})\},\quad m\ge 1,\ n\ge 0.
\end{equation}
Let $\tilde\CA_{h}$ be
given by (\ref{Adef}) with $V_{h}$ replaced by $\tilde V_{h}$. We define
\begin{equation}
  \label{hatmn}
\Tld{\Gl}_{1}(h)=\inf_{\BGf\in\tilde \CA_{h}}\mathfrak{R}_{1}(h,\BGf),\qquad
\tilde\Gl_{m,n}(h)=\inf_{\BGf\in \CF(m,n)}\mathfrak{R}_1(h,\BGf).
\end{equation}
\begin{theorem}
  \label{th:mn}~
\begin{itemize}
\item[(i)]
\begin{equation}
  \label{mninf}
\Tld{\Gl}_{1}(h)=\inf_{\myatop{m\ge 1}{n\ge 0}}\tilde\Gl_{m,n}(h).
\end{equation}
\item[(ii)] The infimum in (\ref{mninf}) is attained at $m=m(h)$ and $n=n(h)$ satisfying
\begin{equation}
  \label{mnbound}
  m(h)\le\frac{C(L)}{\sqrt{h}},\qquad \frac{n(h)^{2}}{m(h)}\le\frac{C(L)}{\sqrt{h}}
\end{equation}
for some constant $C(L)$ depending only on $L$.
\item[(iii)] Let $(m(h),n(h))$ be a minimizer in (\ref{mninf}).
Then the pair $(\CF(m(h),n(h)),\mathfrak{R}_1(h,\BGf))$
  characterizes buckling in the sense of Definition~\ref{def:equiv}.
\end{itemize}
\end{theorem}
\begin{proof}
Part (i). Let
\[
\Ga(h)=\inf_{\myatop{m\ge 1}{n\ge 0}}\tilde\Gl_{m,n}(h).
\]
It is clear that
$\tilde\Gl_{m,n}(h)\ge\Tld{\Gl}_{1}(h)$ for any $m\ge 1$ and $n\ge 0$, since
$\CF(m,n)\subset\tilde \CA_{h}$. Therefore, $\Ga(h)\ge\Tld{\Gl}_{1}(h)$. 

Let us prove the reverse inequality. By definition of $\Ga(h)$ we have
\begin{equation}
  \label{alphamn}
  \int_{\CC_{h}}\av{\Hat{\SFL}_{0}e(\BGf),e(\BGf)}d\Bx\ge\Ga(h)\|\phi_{r,z}\|^{2}
\end{equation}
for any $\BGf\in\CF(m,n)$, and any $m\ge 1$ and $n\ge 0$.
Any $\BGf\in\tilde\CA_{h}$ can be expanded in the Fourier series in $\Gth$ and $z$
\[
\BGf(r,\Gth,z)=\sum_{m=0}^{\infty}\sum_{n=0}^{\infty}\BGf^{(m,n)}(r,\Gth,z),
\]
where $\BGf^{(m,n)}(r,\Gth,z)\in \CF(m,n)$ for all $m\ge 1$, $n\ge 0$. If
$\SFL_{0}$ is isotropic, then the sine and cosine series in $z$ do not couple
and the Plancherel identity implies that the quadratic form
$\av{\Hat{\SFL}_{0}e(\BGf),e(\BGf)}$ diagonalizes in Fourier
space:
\begin{equation}
  \label{Sexp}
\int_{\CC_{h}}\av{\Hat{\SFL}_{0}e(\BGf),e(\BGf)}d\Bx=\sum_{m=0}^{\infty}\sum_{n=0}^{\infty}
\int_{\CC_{h}}\av{\Hat{\SFL}_{0}e(\BGf_{m,n}),e(\BGf_{m,n})}d\Bx.  
\end{equation}
We also have
\[
\|\phi_{r,z}\|^{2}=\sum_{m=0}^{\infty}\sum_{n=0}^{\infty}\|\phi_{r,z}^{(m,n)}\|^{2}=
\sum_{m=1}^{\infty}\sum_{n=0}^{\infty}\|\phi_{r,z}^{(m,n)}\|^{2}.
\]
Inequality (\ref{alphamn}) implies that
\begin{equation}
  \label{mode.by.mode}
\int_{\CC_{h}}\av{\Hat{\SFL}_{0}e(\BGf_{m,n}),e(\BGf_{m,n})}d\Bx\ge\Ga(h)\|\phi_{r,z}^{(m,n)}\|^{2},
\qquad m\ge 1,\ n\ge 0.  
\end{equation}
Summing up, we obtain that
\[
\int_{\CC_{h}}\av{\Hat{\SFL}_{0}e(\BGf),e(\BGf)}d\Bx\ge
\sum_{m=1}^{\infty}\sum_{n=0}^{\infty}
\int_{\CC_{h}}\av{\Hat{\SFL}_{0}e(\BGf_{m,n}),e(\BGf_{m,n})}d\Bx\ge
\Ga(h)\|\phi_{r,z}\|^{2}
\]
for every $\BGf\in\tilde \CA_{h}$. It follows that $\Tld{\Gl}_{1}(h)\ge\Ga(h)$, and
Part (i) is proved.

To establish Part(ii) we require a new delicate Korn-type inequality, proved in
\cite{grha14}. It is a weighted Korn inequality in Nazarov's terminology \cite{naza08}.
\begin{theorem}
\label{th:KtI}
There exists a constant $C(L)$ depending only on $L$ such that
\begin{equation}
\label{KI:first.and.half}
\|(\Grad\BGf)\|^{2}\leq C(L)\left(\frac{\|\Gf_r\|}{h}+\|e(\BGf)\|\right)\|e(\BGf)\|.
\end{equation}
for any $\BGf\in\tilde V_h$.
 \end{theorem}
Observe that, according to the estimate
\[
c(L)h\le\Tld{\Gl}_{1}(h)\le C(L)h.
\]
and Part (i) we have
\[
\inf_{\myatop{m\ge 1}{n\ge 0}}\tilde\Gl_{m,n}(h)=\inf_{(m,n)\in S_{h}}\tilde\Gl_{m,n}(h),
\]
where
\[
S_{h}=\{(m,n):\tilde\Gl_{m,n}(h)\le 2C(L)h\}.
\]
Let us show that the bounds (\ref{mnbound}) hold for all $(m,n)\in
S_{h}$. In particular, the sets $S_{h}$ are finite for all $h>0$, and hence,
the infimum in (\ref{mninf}) is attained. Let $h>0$ and $(m,n)\in S_{h}$ be fixed.
Then, by definition of the infimum there exists
$\BGf^{h}\in \CF(m,n)$ such that $\mathfrak{R}_1(h,\BGf^{h})\le 3C(L)h$.
Hence, there exists a possibly different constant $C(L)$ (not relabeled, but
independent of $m$, $n$ and $h$), such that
\begin{equation}
  \label{optrate}
\|e(\BGf^{h})\|^{2}\le C(L)h\|\phi^{h}_{r,z}\|^{2}=C(L)m^{2}h\|\phi^{h}_{r}\|^{2}.
\end{equation}
To prove the first estimate in (\ref{mnbound}) we apply inequality
(\ref{KI:first.and.half}) to $\BGf^{h}$ and then estimate $\|e(\BGf^{h})\|$
via (\ref{optrate}):
\[
\frac{m^{2}\pi^{2}}{L^{2}}\|\phi^{h}_{r}\|^{2}=\|\phi^{h}_{r,z}\|^{2}\le\|\Grad\BGf^{h}\|^{2}\le
C(L)\left(m^{2}h+\frac{m}{\sqrt{h}}\right)\|\phi^{h}_{r}\|^{2}.
\]
Hence
\[
h+\nth{m\sqrt{h}}\ge c(L)
\]
for some constant $c(L)>0$, independent of $h$. Therefore, we obtain a uniform
in $h\in(0,1)$ upper bound on
$m\sqrt{h}$. To estimate $n$ we write
\[
n^{2}\|\phi^{h}_{r}\|^{2}=\|\phi^{h}_{r,\Gth}\|^{2}\le
C_{0}(\|(\Grad\BGf^{h})_{r\Gth}\|^{2}+\|\phi^{h}_{\Gth}\|^{2}).
\]
By the Poincar\'e inequality
\[
\|\phi^{h}_{\Gth}\|^{2}\le\frac{L^{2}}{\pi^{2}}\|\phi^{h}_{\Gth,z}\|^{2}\le
\frac{L^{2}}{\pi^{2}}\|(\Grad\BGf^{h})_{\Gth z}\|^{2},
\]
and hence
$
n^{2}\|\phi^{h}_{r}\|^{2}\le C(L)\|(\Grad\BGf^{h})\|^{2}.
$
Applying (\ref{KI:first.and.half}) again and estimating $\|e(\BGf^{h})\|$ via (\ref{optrate}) we
obtain
\[
n^{2}\le C(L)\left(hm^{2}+\frac{m}{\sqrt{h}}\right),
\]
from which (\ref{mnbound})$_{2}$ follows via (\ref{mnbound})$_{1}$. Part (ii) is proved now.

Part (iii). Now, let $m(h)$, $n(h)$ be the minimizers in (\ref{mninf}).
It is sufficient to show, due to Lemma~\ref{lem:pairB_hJ},
that $\CF(m(h),n(h))$ contains a buckling mode. By definition of the infimum
in (\ref{hatmn}), for each $h\in(0,h_{0})$ there exists $\BGy_{h}\in
\CF(m(h),n(h))\subset\tilde \CA_{h}$ such that
\[
\Tld{\Gl}_{1}(h)=\Gl_{m(h),n(h)}(h)\le\mathfrak{R}_1(h,\BGy_{h})\le\Tld{\Gl}_{1}(h)
+(\Tld{\Gl}_{1}(h))^{2}.
\]
Therefore,
\[
\lim_{h\to 0}\frac{\mathfrak{R}_1(h,\BGy_{h})}{\Tld{\Gl}_{1}(h)}=1.
\]
Hence, $\BGy_{h}\in \CF(m(h),n(h))$ is a buckling mode, since the pair
$(\tilde \CA_{h},\mathfrak{R}_1(h,\BGf))$ characterizes buckling.
\end{proof}

\subsection{Linearization in $r$}
\label{sub:linr}
In this section we prove that the buckling load and a buckling mode can be
captured by single Fourier harmonics whose $\Gth$ and $z$ components are
linear in $r$. In fact, we specify an explicit structure for buckling mode
candidates. We define the linearization operator as follows:
\[
\CL(\BGf)=(\Gf_{r}(r,\Gth,z),r\Gf_{\Gth}(1,\theta,z)-(r-1)\Gf_{r,\Gth}(1,\theta,z),\Gf_{z}(1,\theta,z)-(r-1)\Gf_{r,z}(1,\theta,z)).
\]
We show now that the buckling mode can be found among the linearized
single Fourier modes
\begin{equation}
  \label{Xlin}
  \CF_{\rm lin}(m,n)=\{\CL(\BGf): \BGf\in \CF(m,n)\},\qquad m\ge 1,\ n\ge 0.
\end{equation}
  \begin{lemma}
  \label{lem:lin}
There exists $C(L)>0$ depending only on $L$, so that for every $h\in(0,1)$,
every $m\ge 1$ and $n\ge 0$, satisfying (\ref{mnbound}), and every $\BGf\in
\CF(m,n)$ we have the estimate
\begin{equation}
  \label{linu}
\mathfrak{R}_{1}(h,\CL(\BGf))\le(1+C(L)h)\mathfrak{R}_{1}(h,\BGf).
\end{equation}
\end{lemma}
\begin{proof}
We will perform linearization in $r$ sequentially, first in
$\phi_{\Gth}$ and then in $\phi_{z}$. 

\textbf{Step 1} (Linearization of $\phi_{\Gth}$.)  We introduce the operator of
linearization of $\phi_{\Gth}$ component.
\[
\CL_{\Gth}(\BGf)=(\phi_{r}(r,\theta,z),r\phi_{\Gth}(1,\theta,z)-(r-1)\phi_{r,\Gth}(1,\theta,z),\phi_{z}(r,\theta,z)),
\]
For $\BGf\in\CF_{\rm lin}(m,n)$ we define
$\BGf^{(1)}=\CL_{\Gth}(\BGf)$. Then, it is easy to see that
$\BGf^{(1)}\in\CF_{\rm lin}(m,n)$. It is clear that
\[
e(\BGf^{(1)})_{rr}=e(\BGf)_{rr},\quad e(\BGf^{(1)})_{zr}=e(\BGf)_{zr},\quad
e(\BGf^{(1)})_{zz}=e(\BGf)_{zz},
\]
Thus we can estimate:
$$\|e(\BGf^{(1)})\|^2\leq \|e(\BGf)\|^2+\|e(\BGf^{(1)})_{r\Gth}\|^2+2\|e(\BGf^{(1)})_{\Gth\Gth}-e(\BGf)_{\Gth\Gth} \|^2+2\|e(\BGf^{(1)})_{\Gth z}-e(\BGf)_{\Gth z} \|^2,$$
$$\|\Trc(e(\BGf^{(1)}))-\Trc(e(\BGf))\|=\|e(\BGf^{(1)})_{\Gth\Gth}-e(\BGf)_{\Gth\Gth} \|^2.$$
We also have
\[
\|e(\BGf^{(1)})_{\Gth\Gth}-e(\BGf)_{\Gth\Gth}\|\le2\|\Gf_{\theta,\theta}^{(1)}-\Gf_{\theta,\theta}\|,\qquad
\|e(\BGf^{(1)})_{\Gth z}-e(\BGf)_{\Gth z}\|\le\|\Gf_{\theta,z}^{(1)}-\Gf_{\theta,z}\|.
\]
Therefore,
 \begin{equation}
 \label{e(A2)e(A)}
 \|e(\BGf^{(1)})\|^{2}\le \|e(\BGf)\|^{2}+\|e(\BGf^{(1)})_{r\Gth}\|^2+
2(\|\Gf_{\theta,\theta}^{(1)}-\Gf_{\theta,\theta}\|^{2}+\|\Gf_{\theta,z}^{(1)}-\Gf_{\theta,z}\|^{2}),
 \end{equation}
and
\begin{equation}
 \label{Tr(A2)Tr(A)}
 \|\Trc(e(\BGf^{(1)}))-\Trc(e(\BGf))\|\leq 2\|\Gf_{\theta,\theta}^{(1)}-\Gf_{\theta,\theta}\|^{2}.
 \end{equation}
Recalling that $\{\BGf,\BGf^{(1)}\}\subset\CF(m,n)$, and that the
inequalities (\ref{mnbound}) imply that $n^{2}\le C(L)/h$, we obtain
\begin{equation}
  \label{step2th}
\|\Gf_{\Gth,\Gth}^{(1)}-\Gf_{\Gth,\Gth}\|^{2}=n^{2}\|\Gf_{\Gth}^{(1)}-\Gf_{\Gth}\|^{2}\le
\frac{C(L)}{h}\|\Gf_{\Gth}^{(1)}-\Gf_{\Gth}\|^{2},
\end{equation}
due to (\ref{mnbound}). Similarly,
\begin{equation}
  \label{step2z}
\|\Gf_{\Gth,z}^{(1)}-\Gf_{\Gth,z}\|^{2}=\frac{\pi^{2} m^{2}}{L^{2}}\|\Gf_{\Gth}^{(1)}-\Gf_{\Gth}\|^{2}\le
\frac{C(L)}{h}\|\Gf_{\Gth}^{(1)}-\Gf_{\Gth}\|^{2},
\end{equation}
Observe that
\[
\|e(\BGf^{(1)})_{r\Gth}\|^2=\|\frac{\Gf_{r,\Gth}-\Gf_{r,\Gth}(1,\Gth,z)}{r}\|^2
=n^{2}\left\|\nth{r}\int_{1}^{r}\phi_{r,r}(t,\Gth,z)dt\right\|^{2}.
\]
Using the inequality 
\begin{equation}
  \label{CS}
\int_{I_{h}}\left(\int_{1}^{r}f(t)dt\right)^{2}dr\le\frac{h^{2}}{4}\int_{I_{h}}f(r)^{2}dr,  
\end{equation}
and the bounds on wave numbers (\ref{mnbound}) we obtain
\[
\|e(\BGf^{(1)})_{r\Gth}\|^2\leq 2n^2C(L)h^{2}\|\Gf_{r,r}\|^2\leq C(L)h\|e(\BGf)\|^{2}.
\]
We now proceed to estimate $\|\Gf_{\Gth}^{(1)}-\Gf_{\Gth}\|$.
Let
\[
w(r,\Gth,z)=\Gf_{\Gth,r}+\Gf_{r,\Gth}-\Gf_{\Gth}=2e(\BGf)_{r\Gth}-(1-r)(\Grad\BGf)_{r\Gth}.
\]
Therefore,
\[
\|w\|^{2}\le8\|e(\BGf)\|^{2}+h^{2}\|\Grad\BGf\|^{2}\le
8\|e(\BGf)\|^{2}+C(L)\sqrt{h}\|e(\BGf)\|^{2}
\]
due to Korn's inequality (\ref{KI}). Thus, $\|w\|\le C(L)\|e(\BGf)\|$.
We can express $\Gf_{\Gth}^{(1)}-\Gf_{\Gth}$ in terms of $w$ as follows
\[
\Gf_{\Gth}-\Gf_{\Gth}^{(1)}=\int_1^rw(t,\Gth,z)dt+\int_1^r(\Gf_\Gth(t,\Gth,z)-\Gf_\Gth(1,\Gth,z))dt-\int_1^r(\Gf_{r,\Gth}(t,\Gth,z)-\Gf_{r,\Gth}(1,\Gth,z))dt.
\]
Hence, by (\ref{CS}), we have
\[
\|\Gf_{\Gth}-\Gf_{\Gth}^{(1)}\|^{2}\le \frac{3h^{2}}{4}(\|w\|^{2}+\|\Gf_{\Gth}-\Gf_\Gth(1,\Gth,z)\|^{2}+\|\Gf_{r,\Gth}-\Gf_{r,\Gth}(1,\Gth,z)\|^{2}).
\]
By the Poincar\'e inequality followed by the application of Korn's
inequality (\ref{KI}) we obtain,
\[
\|\Gf_{\Gth}-\Gf_\Gth(1,\Gth,z)\|^{2}\le h^{2}\|\Gf_{\Gth,r}\|^{2}\le C(L)\sqrt{h}\|e(\BGf)\|^{2}.
\]
Similarly, by the Poincar\'e inequality and (\ref{mnbound}) we estimate
\[
\|\Gf_{r,\Gth}-\Gf_{r,\Gth}(1,\Gth,z)\|^{2}=n^2\|\Gf_{r}-\Gf_{r}(1,\Gth,z)\|^{2}\leq C(L)n^2h^2\|\Gf_{r,r}\|^2\leq C(L)h\|e(\BGf)\|^{2}.
\]
We conclude that
\[
\|\Gf_{\Gth}-\Gf_{\Gth}^{(1)}\|^{2}\le C(L)h^{2}\|e(\BGf)\|^{2}.
\]
Hence, (\ref{e(A2)e(A)}) and (\ref{Tr(A2)Tr(A)}) become respectively,
\begin{equation}
  \label{u1u2}
\|e(\BGf^{(1)})\|^{2}\le \|e(\BGf)\|^{2}(1+C(L)h),
\end{equation}
 and
\begin{equation}
  \label{Tru1u2}
\|\Trc(e(\BGf^{(1)}))\|^{2}\le\|\Trc(e(\BGf))\|^{2}+C(L)h\|e(\BGf)\|^{2}.
\end{equation}
Hence, by  (\ref{u1u2}), (\ref{Tru1u2}) and the coercivity of $\Hat{\SFL}_{0}$,
we have
\begin{multline}
  \label{step2}
\int_{\CC_{h}}\av{\Hat{\SFL}_{0}e(\BGf^{(1)}),e(\BGf^{(1)})}d\Bx=\frac{1}{1+\nu}\left(
\frac{\nu}{1-2\nu}\|\Trc(e(\BGf^{(1)}))\|^2+\|e(\BGf^{(1)})\|^2\right)\le\\
(1+C(L)h)\int_{\CC_{h}}\av{\Hat{\SFL}_{0}e(\BGf),e(\BGf)}d\Bx.
\end{multline}

\textbf{Step 2} (Linearization of $\phi_{z}$.)  In this step we define
$
\BGf^{(2)}=\CL(\BGf)=\CL(\BGf^{(1)}),
$
and proceed exactly as in Step 1. We observe that 
\[
e(\BGf^{(2)})_{rr}=e(\BGf^{(1)})_{rr},\qquad
e(\BGf^{(2)})_{r\Gth}=e(\BGf^{(1)})_{r\Gth},\qquad
e(\BGf^{(2)})_{\Gth\Gth}=e(\BGf^{(1)})_{\Gth\Gth},
\] 
and hence,
\[
\|e(\BGf^{(2)})\|^2\leq \|e(\BGf^{(1)})\|^2+\|e(\BGf^{(2)})_{rz}\|^2+2\|e(\BGf^{(2)})_{\Gth z}-e(\BGf^{(1)})_{\Gth z}\|^2+2\|e(\BGf^{(2)})_{zz}-e(\BGf^{(1)})_{zz}\|^2,
\]
and
\begin{equation}
 \label{Tr(A2)Tr(A3)}
 \|\Trc(e(\BGf^{(1)}))-\Trc(e(\BGf^{(2)}))\|\leq 2\|\Gf_{z,z}^{(1)}-\Gf_{z,z}^{(2)}\|^{2}.
 \end{equation}
We also have
\begin{equation}
  \label{zztherr}
  \|e(\BGf^{(2)})_{\Gth z}-e(\BGf^{(1)})_{\Gth z}\|\le 2\|\Gf^{(1)}_{z,\Gth}-\Gf^{(2)}_{z,\Gth}\|,\qquad
\|e(\BGf^{(2)})_{zz}-e(\BGf^{(1)})_{zz}\|\le \|\Gf^{(1)}_{z,z}-\Gf^{(2)}_{z,z}\|.
\end{equation}
For functions $\{\BGf^{(1)},\BGf^{(2)}\}\subset\CF(m,n)$ we obtain
\[
\|\Gf^{(1)}_{z,\Gth}-\Gf^{(2)}_{z,\Gth}\|=n\|\Gf^{(1)}_{z}-\Gf^{(2)}_{z}\|\le\frac{C(L)}{h}\|\Gf^{(1)}_{z}-\Gf^{(2)}_{z}\|,
\]
and
\[
\|\Gf^{(1)}_{z,z}-\Gf^{(2)}_{z,z}\|=\frac{\pi
  m}{L}\|\Gf^{(1)}_{z}-\Gf^{(2)}_{z}\|\le\frac{C(L)}{h}\|\Gf^{(1)}_{z}-\Gf^{(2)}_{z}\|,
\]
where the bounds (\ref{mnbound}) on wave numbers have been used. Hence,
\begin{equation}
  \label{step3thz}
  \|e(\BGf^{(2)})_{\Gth z}-e(\BGf^{(1)})_{\Gth z}\|^{2}\le\frac{C(L)}{h}\|\Gf^{(1)}_{z}-\Gf^{(2)}_{z}\|^{2},\quad
\|e(\BGf^{(2)})_{zz}-e(\BGf^{(1)})_{zz}\|^{2}\le\frac{C(L)}{h}\|\Gf^{(1)}_{z}-\Gf^{(2)}_{z}\|^{2}.
\end{equation}
For $\|e(\BGf^{(2)})_{rz}\|$ we obtain
\[
\|e(\BGf^{(2)})_{rz}\|^2=\|\Gf_{r,z}-\Gf_{r,z}(1,\Gth,z)\|^2
=\frac{\pi^2m^2}{L^2}\|\Gf_{r}-\Gf_{r}(1,\Gth,z)\|^2=
\frac{\pi^2m^2}{L^2}\left\|\int_{1}^{r}\Gf_{r,r}(t,\Gth,z)dt\right\|^2.
\]
Applying inequalities (\ref{CS}) and
(\ref{mnbound}) we obtain
\[
\|e(\BGf^{(2)})_{rz}\|^2\leq C(L)m^2h^2\|\Gf_{r,r}\|^2\leq C(L)h\|e(\BGf^{(1)})\|^{2}.
\]
Finally, we estimate the norm $\|\Gf^{(1)}_{z}-\Gf^{(2)}_{z}\|.$
Integrating the identity $\Gf_{z,r}^{(1)}=2e(\BGf^{(1)})_{rz}-\Gf_{r,z}^{(1)}$ we get
\[
\Gf_z^{(1)}(r,\Gth,z)-\Gf_z^{(1)}(1,\Gth,z)=
2\int_{1}^r e(\BGf^{(1)})_{rz}(t,\Gth,z)dt-\int_{1}^r\Gf_{r,z}^{(1)}(t,\Gth,z)dt.
\]
Therefore,
\[
\Gf_{z}^{(1)}-\Gf^{(2)}_{z}=2\int_{1}^r e(\BGf^{(1)})_{rz}(t,\Gth,z)dt
-\int_{1}^r(\Gf_{r,z}^{(1)}(t,\Gth,z)-\Gf_{r,z}^{(1)}(1,\Gth,z))dt.
\]
Applying inequalities (\ref{CS}) and (\ref{mnbound}) we get
\begin{multline*}
\|\Gf_{z}^{(1)}-\Gf^{(2)}_{z}\|^2 \leq h^2(\|e(\BGf^{(1)})\|^2+\|\Gf_{r,z}^{(1)}(r,\Gth,z)-\Gf_{r,z}^{(1)}(1,\Gth,z)\|^2)=\\
h^2(\|e(\BGf^{(1)})\|^2+\frac{\pi^2m^2}{L^2}\left\|\int_{1}^{r}\Gf_{r,r}^{(1)}(t,\Gth,z)dt\right\|^2)
\leq h^2(\|e(\BGf^{(1)})\|^2+\frac{\pi^2m^2h^{2}}{L^2}\|\Gf_{r,r}^{(1)}\|^2)
\leq\\ h^2(1+\frac{\pi^2m^2h^{2}}{L^2})\|e(\BGf^{(1)})\|^2\le C(L)h^{2}\|e(\BGf^{(1)})\|^2.
\end{multline*}
Applying this estimate to
(\ref{step3thz}) and (\ref{Tr(A2)Tr(A3)}) we obtain
\[
\|e(\BGf^{(2)})_{\Gth z}-e(\BGf^{(1)})_{\Gth z}\|^{2}\le C(L)h\|e(\BGf^{(1)})\|^{2},\qquad
\|e(\BGf^{(2)})_{zz}-e(\BGf^{(1)})_{zz}\|^{2}\le C(L)h\|e(\BGf^{(1)})\|^{2},
\]
and
$$
\|\Trc(e(\BGf^{(2)}))\|^{2}\le\|\Trc(e(\BGf^{(1)}))\|^{2}+C(L)h\|e(\BGf^{(1)})\|^{2}.
$$
We conclude that
\[
\|e(\BGf^{(2)})\|^{2}\le \|e(\BGf^{(1)})\|^{2}(1+C(L)h),\qquad
\|\Trc(\BGf^{(2)})\|^{2}\le\|\Trc(e(\BGf^{(1)}))\|^{2}+C(L)h\|e(\BGf^{(1)})\|^{2},
\]
and hence, by coercivity of $\Hat{\SFL}_{0}$ we have
\begin{equation}
  \label{step3}
\int_{\CC_{h}}\av{\Hat{\SFL}_{0}e(\BGf^{(2)}),e(\BGf^{(2)})}d\Bx\leq
(1+C(L) h)\int_{\CC_{h}}\av{\Hat{\SFL}_{0}e(\BGf^{(1)}),e(\BGf^{(1)})}d\Bx.
\end{equation}
Combining (\ref{step2}) and  (\ref{step3}) we obtain
(\ref{linu}).
\end{proof}
Lemma~\ref{lem:lin} permits us to look for a buckling mode among those
single Fourier modes, whose $\Gth$ and $z$ components are linear in $r$. Let $C(L)$
be a constant, whose existence is guaranteed by Lemma~\ref{lem:lin}. Let 
\[
\CM_{h}=\{(m,n): n\ge 0,\ m\ge 1\text{ and inequalities (\ref{mnbound})  hold}\}.
\]
Let
\[
\CF_{\rm lin}^{h}=\bigcup_{(m,n)\in\CM_{h}}\CF_{\rm lin}(m,n).
\]
\begin{corollary}
\label{cor:main}
The pair $(\CF_{\rm lin}^{h},\mathfrak{R}_{1})$ characterizes
buckling.
\end{corollary}
\begin{proof}
By Lemma~\ref{lem:pairB_hJ} it is sufficient to show that $\CF_{\rm lin}^{h}$ 
contains a buckling mode. Let $(m(h),n(h))$ be minimizers in
(\ref{mninf}). Then, according to Theorem~\ref{th:mn}, $(m(h),n(h))\in\CM_{h}$
and $\CF(m(h),n(h))$ contains a buckling mode. Let $\BGy_{h}\in \CF(m(h),n(h))$ be a
buckling mode. Let us show that $\CL(\BGy_{h})\in \CF_{\rm lin}^{h}$
is also a buckling mode. Indeed, by Lemma~\ref{lem:lin}
\[
1\le\frac{\mathfrak{R}_{1}(h,\CL(\BGy_{h}))}{\Tld{\Gl}_{1}(h)}\le
(1+C(L)h)\frac{\mathfrak{R}_{1}(h,\BGy_{h})}{\Tld{\Gl}_{1}(h)}.
\]
Taking a limit as $h\to 0$ and using the fact that $\BGy_{h}$ is a buckling
mode, we obtain
\[
\lim_{h\to 0}\frac{\mathfrak{R}_{1}(h,\CL(\BGy_{h}))}{\Tld{\Gl}_{1}(h)}=1.
\]
Hence, $\CL(\BGy_{h})$ is also a buckling mode, since, by Theorem~\ref{th:mn},
the pair $(\CF(m(h), n(h)),\mathfrak{R}_{1}(h,\BGf))$ characterizes buckling.
\end{proof}

\subsection{Simplification via buckling equivalence}
\label{sub:alg}
The linearization Lemma~\ref{lem:lin} allowed us to reduce the set of buckling
modes significantly. Yet, even for functions $\BGf\in \CF_{\rm lin}(m,n)$ the
explicit representation of the functional $\mathfrak{R}_{1}(h,\BGf)$ is
extremely messy. This can be dealt with by further simplification of the
functional via buckling equivalence that permits us to eliminate lower order
terms that do not influence the asymptotic behavior of the functional.  Our
first step is to simplify the denominator in $\mathfrak{R}_{1}(h,\BGf)$ by
replacing the unknown function $f_{r}(r)$ in
$\phi_{r}=f_r(r)\cos(mz)e^{in\Gth}$ with $f_{r}(1)$. Here, in order to
simplify the formulas we use $m$ in place of $\pi m/L$. Hence, we define a new
simplified functional
\[
\mathfrak{R}_{2}(h,\BGf)=\frac{\int_{\CC_{h}}\av{\Hat{\SFL}_{0}e(\BGf),e(\BGf)}d\Bx}
{\int_{\CC_{h}}|\phi_{r,z}(1,\Gth,z)|^{2}d\Bx}.
\]
\begin{lemma}
  \label{lem:simpdenom}
The functionals $\mathfrak{R}_{1}(h,\BGf)$ and $\mathfrak{R}_{2}(h,\BGf)$
are buckling equivalent.
\end{lemma}
\begin{proof}
We observe that
\[
|\phi_{r,z}(r,\Gth,z)-\phi_{r,z}(1,\Gth,z)|=m\left|\int_{1}^{r}\phi_{r,r}(t,\Gth,z)\right|.
\]
Hence, due to (\ref{CS})
\[
\|\phi_{r,z}(r,\Gth,z)-\phi_{r,z}(1,\Gth,z)\|\le mh\|e(\BGf)\|.
\]
Therefore,
\[
\left|\int_{\CC_{h}}|\phi_{r,z}(r,\Gth,z)|^{2}d\Bx-
\int_{\CC_{h}}|\phi_{r,z}(1,\Gth,z)|^{2}d\Bx\right|\le
mh\|e(\BGf)\|\|\phi_{r,z}\|\le m\sqrt{h}\|e(\BGf)\|^{2},
\]
due to Theorem~\ref{th:KI}. Hence,
\[
\left|\nth{\mathfrak{R}_{1}(h,\BGf)}-\nth{\mathfrak{R}_{2}(h,\BGf)}\right|\le
Cm\sqrt{h},
\]
by coercivity of $\SFL_{0}$. For $(m,n)\in\CM_{h}$ we conclude that, due to
(\ref{perfect.lambda.hat.upper}) and (\ref{mnbound}),
\[
\lim_{h\to
  0}\Gl(h)\left|\nth{\mathfrak{R}_{1}(h,\BGf)}-\nth{\mathfrak{R}_{2}(h,\BGf)}\right|=0.
\]
Theorem~\ref{th:Bequivalence} applies and hence the functionals
$\mathfrak{R}_{1}(h,\BGf)$ and $\mathfrak{R}_{2}(h,\BGf)$ are buckling
equivalent.
\end{proof}
We can also simplify the numerator of $\mathfrak{R}_{2}(h,\BGf)$ by replacing
$r$ with 1 in those places, where it does not affect the asymptotics. The
simplification now proceeds at the level of individual
components of $e(\BGf)$. We may, \WLOG, restrict our attention to
$\BGf\in\CF_{\rm lin}(m,n)$, such that
\begin{equation}
  \label{pbmr}
\phi_{r}=f_{r}(r)\cos(n\Gth)\cos(mz).  
\end{equation}
Of course, choosing $\sin(n\Gth)$ instead of $\cos(n\Gth)$ in (\ref{pbmr}) works just
as well. The choice between $\sin(n\Gth)$ and $\cos(n\Gth)$ in the remaining components becomes
uniquely determined by the requirement that every entry in $e(\BGf)$ must be made up of
terms that have the same kind of trigonometric function in $n\Gth$. (We have
already taken care of the same requirement in $z$.) Hence, the $\Gth$ and $z$
components of $\BGf\in\CF_{\rm lin}(m,n)$ must have the form
\begin{equation}
  \label{pbmthz}
  \begin{cases}
  \phi_{\Gth}=(ra_{\Gth}+(r-1)nf_{r}(1))\sin(n\Gth)\cos(mz),\\
  \phi_{z}=(a_{z}+(r-1)mf_{r}(1))\cos(n\Gth)\sin(mz),
\end{cases}
\end{equation}
where $a_{\Gth}$ and $a_{z}$ are real scalars that determine the amplitude of the
Fourier modes. We compute,
\[
\begin{cases}
e(\BGf)_{rr}=f'_{r}(r)\cos(n\Gth)\cos(mz),\\[2ex]
e(\BGf)_{r\Gth}=\dfrac{n(f_{r}(1)-f_{r}(r))}{2r}\sin(n\Gth)\cos(mz),\\[2ex]
e(\BGf)_{rz}=\dfrac{m(f_{r}(1)-f_{r}(r))}{2}\cos(n\Gth)\sin(mz),\\[2ex]  
e(\BGf)_{\Gth\Gth}=\dfrac{n(ra_{\Gth}+(r-1)nf_{r}(1))+f_{r}(r)}{r}\cos(n\Gth)\cos(mz),\\[2ex]
e(\BGf)_{\Gth z}=-\dfrac{mr^{2}a_{\Gth}+na_{z}+(r^{2}-1)mnf_{r}(1)}{2r}\sin(n\Gth)\sin(mz),\\[2ex]
e(\BGf)_{rz}=m(a_{z}+(r-1)mf_{r}(1))\cos(n\Gth)\cos(mz).
\end{cases}
\]
We can now replace $e(\BGf)$ with a much simpler matrix $E(\BGf)$, given by
\[
\begin{cases}
E(\BGf)_{rr}=\dfrac{f'_{r}(r)}{\sqrt{r}}\cos(n\Gth)\cos(mz),\\[2ex]
E(\BGf)_{r\Gth}=0,\\[2ex]
E(\BGf)_{rz}=0,\\[2ex]  
E(\BGf)_{\Gth\Gth}=\dfrac{n(ra_{\Gth}+(r-1)nf_{r}(1))+f_{r}(1)}{\sqrt{r}}\cos(n\Gth)\cos(mz),\\[2ex]
E(\BGf)_{\Gth z}=-\dfrac{mr^{2}a_{\Gth}+na_{z}+(r^{2}-1)mnf_{r}(1)}{2\sqrt{r}}\sin(n\Gth)\sin(mz),\\[2ex]
E(\BGf)_{rz}=\dfrac{m(a_{z}+(r-1)mf_{r}(1))}{\sqrt{r}}\cos(n\Gth)\cos(mz)
\end{cases}
\]
\begin{lemma}
  \label{lem:finsimp}
The functionals $\mathfrak{R}_{2}(h,\BGf)$ and
\begin{equation}
  \label{Rstar}
\mathfrak{R}_{3}(h,\BGf)=\frac{\int_{\CC_{h}}\av{\Hat{\SFL}_{0}E(\BGf),E(\BGf)}d\Bx}
{\int_{\CC_{h}}|\phi_{r,z}(1,\Gth,z)|^{2}d\Bx}  
\end{equation}
are buckling equivalent.
\end{lemma}
\begin{proof}
  Observing that 
\[
f_{r}(r)-f_{r}(1)=\int_{1}^{r}f'_{r}(t)dt,
\]
we obtain via (\ref{CS}) that
\[
\|e(\BGf)_{r\Gth}\|^{2}\le Cn^{2}h^{2}\|f'_{r}\|^{2}\le
Cn^{2}h^{2}\|e(\BGf)_{rr}\|^{2}.
\]
Similarly,
\[
\|e(\BGf)_{rz}\|^{2}\le Cm^{2}h^{2}\|e(\BGf)_{rr}\|^{2}.
\]
Hence, for every $(m,n)\in\CM_{h}$ we have
\[
\|e(\BGf)_{r\Gth}\|^{2}+\|e(\BGf)_{rz}\|^{2}\le Ch\|e(\BGf)_{rr}\|^{2}.
\]
For the components $(rr)$, $(\Gth z)$ and $(zz)$ we have
\[
E(\BGf)_{rr}=\frac{e(\BGf)_{rr}}{\sqrt{r}},\quad
E(\BGf)_{\Gth z}=\sqrt{r}e(\BGf)_{\Gth z},\quad
E(\BGf)_{zz}=\frac{e(\BGf)_{zz}}{\sqrt{r}}.
\]
Therefore,
\[
|E(\BGf)_{rr}-e(\BGf)_{rr}|\le Ch|e(\BGf)_{rr}|,\quad
|E(\BGf)_{\Gth z}-e(\BGf)_{\Gth z}|\le Ch|e(\BGf)_{\Gth z}|,\quad
|E(\BGf)_{zz}-e(\BGf)_{zz}|\le Ch|e(\BGf)_{zz}|.
\]
Finally we compute
\[
E(\BGf)_{\Gth\Gth}-e(\BGf)_{\Gth\Gth}=(\sqrt{r}-1)e(\BGf)_{\Gth\Gth}-
\frac{f_{r}(r)-f_{r}(1)}{\sqrt{r}}\cos(n\Gth)\cos(mz),
\]
which implies
\[
\|E(\BGf)_{\Gth\Gth}-e(\BGf)_{\Gth\Gth}\|\le
Ch(\|e(\BGf)_{\Gth\Gth}\|+\|e(\BGf)_{rr}\|).
\]
We conclude that that
\[
\|E(\BGf)-e(\BGf)\|\le C\sqrt{h}\|e(\BGf)\|,
\]
and thus
\[
\left|\int_{\CC_{h}}\av{\Hat{\SFL}_{0}E(\BGf),E(\BGf)}d\Bx-
\int_{\CC_{h}}\av{\Hat{\SFL}_{0}e(\BGf),e(\BGf)}d\Bx\right|\le
C\sqrt{h}\|e(\BGf)\|^{2}\le C\sqrt{h}\int_{\CC_{h}}\av{\Hat{\SFL}_{0}e(\BGf),e(\BGf)}d\Bx,
\]
by coercivity of $\Hat{\SFL}_{0}$. It follows that
\[
|\mathfrak{R}_{3}(h,\BGf)-\mathfrak{R}_{2}(h,\BGf)|\le
C\sqrt{h}\mathfrak{R}_{2}(h,\BGf)\le C\sqrt{h}\mathfrak{R}_{3}(h,\BGf)+C\sqrt{h}
|\mathfrak{R}_{3}(h,\BGf)-\mathfrak{R}_{2}(h,\BGf)|.
\]
Thus,
\[
|\mathfrak{R}_{3}(h,\BGf)-\mathfrak{R}_{2}(h,\BGf)|\le\frac{C\sqrt{h}}{1-C\sqrt{h}}\mathfrak{R}_{3}(h,\BGf).
\]
Dividing this inequality by $\mathfrak{R}_{2}(h,\BGf)\mathfrak{R}_{3}(h,\BGf)$
we obtain
\[
\left|\nth{\mathfrak{R}_{2}(h,\BGf)}-\nth{\mathfrak{R}_{3}(h,\BGf)}\right|\le
\frac{C\sqrt{h}}{\mathfrak{R}_{2}(h,\BGf)}.
\]
Therefore,
\[
\sup_{\BGf\in\CF_{\rm lin}^{h}}\tilde\Gl(h)\left|\nth{\mathfrak{R}_{2}(h,\BGf)}-
\nth{\mathfrak{R}_{3}(h,\BGf)}\right|\le 
\frac{C\tilde\Gl(h)\sqrt{h}}{\displaystyle\inf_{\BGf\in\CF_{\rm lin}^{h}}\mathfrak{R}_{2}(h,\BGf)}.
\]
It follows that, due to (\ref{perfect.lambda.hat.upper}),
\[
\lim_{h\to 0}\sup_{\BGf\in\CF_{\rm lin}^{h}}\Gl(h)\left|\nth{\mathfrak{R}_{2}(h,\BGf)}-
\nth{\mathfrak{R}_{3}(h,\BGf)}\right|=0.
\]
The application of Theorem~\ref{th:Bequivalence} completes the proof.
\end{proof}

\subsection{The formula for the buckling load}
At this point the strategy for finding the asymptotic formula for the buckling load can
be stated as follows. We first compute
\begin{equation}
  \label{almost}
  \Gl_{3}(h;m,n)=\inf_{\BGf\in \CF_{\rm lin}(m,n)}\mathfrak{R}_{3}(h,\BGf),
\end{equation}
and then we find $m(h)$ and $n(h)$ as minimizers in
\begin{equation}
  \label{ldlin}
  \Gl_{3}(h)=\min_{\myatop{m\ge 1}{n\ge 0}}\Gl_{3}(h;m,n).
\end{equation}
The goal of the section is to prove that
\begin{equation}
  \label{goal}
\lim_{h\to 0}\frac{\Gl_{3}(h)}{\Gl^{*}(h)}=1,\qquad  \Gl^{*}(h)=\frac{h}{\sqrt{3(1-\nu^{2})}}.
\end{equation}

The functional $\mathfrak{R}_{3}(h,\BGf)$ given by (\ref{Rstar}) will now be
analyzed in its explicit form.
\begin{multline*}
  \mathfrak{R}_{3}(h,\BGf)=\nth{2(\nu+1)hm^{2}|f_{r}(1)|^{2}}\int_{I_{h}}\{
(mr^{2}a_{\Gth}+na_{z}+(r^{2}-1)mnf_{r}(1))^{2}+\\
+2(f'_{r})^{2}+2(nra_{\Gth}+(r-1)n^{2}f_{r}(1)+f_{r}(1))^{2}+2m^{2}(a_{z}+(r-1)mf_{r}(1))^{2}\\
\GL(f'_{r}(r)+nra_{\Gth}+(r-1)n^{2}f_{r}(1)+f_{r}(1)+ma_{z}+(r-1)m^{2}f_{r}(1))^{2}\}dr,
\end{multline*}
where $\GL=\frac{2\nu}{1-2\nu}$.
We minimize the numerator in $f_{r}(r)$ with prescribed value $f_{r}(1)$. This
can be done by minimizing the numerator in $f'_{r}(r)$ treating it as a scalar
variable for each fixed $r$:
$$
f'_{r}(r)=-\frac{\GL}{\GL+2}p(r),
$$
where
\[
p(r)=nra_{\Gth}+(r-1)n^{2}f_{r}(1)+f_{r}(1)+ma_{z}+(r-1)m^{2}f_{r}(1).
\]
Thus, we reduce the problem of computing $\Gl_{3}(h;m,n)$ to
finite-dimensional unconstrained minimization:
\begin{equation}
  \label{fdmin}
  \Gl_{3}(h;m,n)=\min_{a_{\Gth},a_{z},f_{r}(1)}
\frac{\int_{I_{h}}\{\frac{2\GL}{\GL+2}p(r)^{2}+q(r)\}dr}{2(\nu+1)hm^{2}|f_{r}(1)|^{2}},
\end{equation}
where
\begin{multline*}
q(r)=(mr^{2}a_{\Gth}+na_{z}+(r^{2}-1)mnf_{r}(1))^{2}+
2(nra_{\Gth}+(r-1)n^{2}f_{r}(1)+f_{r}(1))^{2}+\\
2m^{2}(a_{z}+(r-1)mf_{r}(1))^{2}.  
\end{multline*}
Since the function to be minimized in (\ref{fdmin}) is homogeneous of degree
zero in the vector variable $(a_{\Gth},a_{z},f_{r}(1))$, we can set
$f_r(1)=1$, \WLOG. Then, evaluating the integral in $r$ we obtain
\[
\Gl_{3}(h;m,n)=\min_{a_{\Gth},a_{z}}\nth{2(\nu+1)m^{2}}\left\{
Q^{(0)}_{m,n}(a_{\Gth},a_{z})+\frac{h^{2}}{12}Q^{(1)}_{m,n}(a_{\Gth},a_{z})
+\frac{h^{4}}{80}Q^{(2)}_{m,n}(a_{\Gth},a_{z})\right\},
\]
where
\[
Q^{(0)}_{m,n}=\frac{2\GL}{\GL+2}(1+na_{\Gth}+ma_{z})^{2}+2(na_{\Gth}+1)^{2}+
2m^{2}a_{z}^{2}+(ma_{\Gth}+na_{z})^{2},
\]
\[
Q^{(1)}_{m,n}=\frac{2\GL}{\GL+2}(na_{\Gth}+m^{2}+n^{2})^{2}+2n^{2}(a_{\Gth}+n)^{2}+
2m^{4}+4m^{2}(a_{\Gth}+n)^{2}+2m(a_{\Gth}+n)(ma_{\Gth}+na_{z}),
\]
\[
Q^{(2)}_{m,n}=m^{2}(a_{\Gth}+n)^{2}.
\]
Let us show that the last term in $Q^{(1)}_{m,n}$, as well as
$Q^{(2)}_{m,n}$ can be discarded.
Let
\[
\tilde Q^{(1)}_{m,n}(a_{\Gth})=\frac{2\GL}{\GL+2}(na_{\Gth}+m^{2}+n^{2})^{2}+2n^{2}(a_{\Gth}+n)^{2}+
2m^{4}+4m^{2}(a_{\Gth}+n)^{2}
\]
be the simplified version of $Q^{(1)}_{m,n}$. We observe that
\[
2m|(a_{\Gth}+n)(ma_{\Gth}+na_{z})|\le
hm^{2}(a_{\Gth}+n)^{2}+\nth{h}(ma_{\Gth}+na_{z})^{2}\le \frac{h}{4}\tilde Q^{(1)}_{m,n}+\nth{h}Q^{(0)}_{m,n}.
\]
Therefore,
\[
\frac{h^{4}}{80}m^{2}(a_{\Gth}+n)^{2}+\frac{h^{2}}{6}m|(a_{\Gth}+n)(ma_{\Gth}+na_{z})|\le
(h^{2}+h)\left(Q^{(0)}_{m,n}+\frac{h^{2}}{12}\tilde Q^{(1)}_{m,n}\right).
\]
Hence,
\[
(1-h-h^{2})\left(Q^{(0)}_{m,n}+\frac{h^{2}}{12}\tilde Q^{(1)}_{m,n}\right)\le Q^{(0)}_{m,n}+\frac{h^{2}}{12}Q^{(1)}_{m,n}\le(1+h+h^{2})\left(Q^{(0)}_{m,n}+\frac{h^{2}}{12}\tilde Q^{(1)}_{m,n}\right)
\]
If we denote 
\begin{equation}
  \label{finlmbda}
\tilde\Gl_{3}(h;m,n)=\min_{a_{\Gth},a_{z}}\nth{2(\nu+1)m^{2}}\left\{
Q^{(0)}_{m,n}(a_{\Gth},a_{z})+\frac{h^{2}}{12}\tilde Q^{(1)}_{m,n}(a_{\Gth})\right\},  
\end{equation}
then
\[
(1-h-h^{2})\tilde\Gl_{3}(h;m,n)\le\Gl_{3}(h;m,n)\le(1+h+h^{2})\tilde\Gl_{3}(h;m,n),
\]
which implies that
\begin{equation}
  \label{l32l3pr}
\lim_{h\to 0}\frac{\tilde\Gl_{3}(h)}{\Gl_{3}(h)}=1,\qquad
\tilde\Gl_{3}(h)=\min_{\myatop{m\ge 1}{n\ge 0}}\Gl_{3}(h;m,n).  
\end{equation}
Minimizing $Q^{(0)}_{m,n}(a_{\Gth},a_{z})$ in $a_{z}$ we obtain
\begin{equation}
\label{formula.a_z}
a_{z}=-\frac{m(2\nu+(\nu+1)na_{\Gth})}{2m^{2}+(1-\nu)n^{2}}.
\end{equation}
The minimization in $a_{\Gth}$ was too tedious to be done by hand. Using
computer algebra software (Maple), we have obtained
the following expression for $\tilde\Gl_{3}(h;m,n)$:
\begin{equation}
  \label{l3pr}
\tilde\Gl_{3}(h;m,n)=\frac{m^{2}(1-\nu^{2})+H(m^{2}+n^{2})^{4}+Hr_{1}(m,n)+H^{2}r_{2}(m,n)}
{(1-\nu^{2})m^{2}(m^{2}+n^{2})^{2}+Hm^{2}r_{3}(m,n)},\quad H=\frac{h^{2}}{12},
\end{equation}
where $r_{1}(m,n)$ is a polynomial in $(m,n)$ of degree 6, $r_{2}(m,n)$ is a
polynomial in $(m,n)$ of degree 8 and $r_{3}(m,n)$ is a
polynomial in $(m,n)$ of degree 4. The minimum was achieved at
\begin{equation}
\label{formula.a_za_gth}
a_{\Gth}=-\frac{n(n^{2}+(\nu+2)m^{2})+Hs_{1}(m,n)}{(m^{2}+n^{2})^{2}+Hs_{2}(m,n)},
\end{equation}
where $s_{1}(m,n)$ is a polynomial in $(m,n)$ of degree 5, and
$s_{2}(m,n)$ is a polynomial in $(m,n)$ of degree 4.
Let us show that the terms $r_{i}(m,n)$ do not affect the asymptotics of
$\tilde\Gl_{3}(h)$.
Let
\begin{equation}
  \label{l3st}
\Gl^{*}_{3}(h;m,n)=\frac{m^{4}(1-\nu^{2})+H(m^{2}+n^{2})^{4}}
{(1-\nu^{2})m^{2}(m^{2}+n^{2})^{2}},\qquad \Gl^{*}_{3}(h)=\min_{\myatop{m\ge 1}{n\ge  0}}\Gl^{*}_{3}(h;m,n).
\end{equation}
\begin{lemma}
  \label{lem:final}
  \begin{equation}
    \label{l3st.formula}
    \Gl^{*}_{3}(h)=\frac{h}{\sqrt{3(1-\nu^{2})}},
  \end{equation}
and is attained on the Koiter circle \cite{koit45}
\begin{equation}
\label{formula.m.n.b.l}
\frac{m}{m^{2}+n^{2}}=\sqrt{\frac{\Gl_{3}^{*}(h)}{2}},\qquad m\ge \frac{\pi}{L},\ n\ge 0.
\end{equation}
Moreover,
\begin{equation}
  \label{lim3}
\lim_{h\to 0}\frac{\tilde\Gl_{3}(h)}{\Gl^{*}_{3}(h)}=1,  
\end{equation}
\end{lemma}
\begin{proof}
  Formulas (\ref{l3st.formula}) and (\ref{formula.m.n.b.l}) become
  obvious, if we observe that
\[ 
\Gl^{*}_{3}(h;m,n)=\frac{m^{2}}{(m^{2}+n^{2})^{2}}+\frac{H(m^{2}+n^{2})^{2}}{(1-\nu^{2})m^{2}}.
\]
It is also clear from the degrees of polynomials $r_{2}(m,n)$ and $r_{3}(m,n)$ that
\[
\sup_{\myatop{m\ge \pi/L}{n\ge  0}}\frac{H^{2}r_{2}(m,n)}{2m^{4}(1-\nu^{2})+2H(m^{2}+n^{2})^{4}}\le
\sup_{\myatop{m\ge \pi/L}{n\ge 0}}\frac{Hr_{2}(m,n)}{2(m^{2}+n^{2})^{4}}\le CH,
\]
and
\[
\sup_{\myatop{m\ge \pi/L}{n\ge 0}}\frac{Hr_{3}(m,n)}{(1-\nu)(m^{2}+n^{2})^{2}}\le CH,
\]
for some constant $C$, independent of $m$, $n$, and $h$. 

In order to show that we can also eliminate $Hr_{1}(m,n)$ from the numerator
of $\tilde\Gl_{3}(h;m,n)$ we observe that for any constant $C$
\[
\lim_{h\to 0}\min_{m^{2}+n^{2}\le C}\tilde\Gl_{3}(h;m,n)>0.
\]
Hence, if $(m(h),n(h))$ is a minimizer in (\ref{l3pr}), then
$m(h)^{2}+n(h)^{2}\to\infty$, as $h\to 0$.  If $(m^{*}(h),n^{*}(h))$ denotes a
minimizer in (\ref{l3st}), then formulas (\ref{l3st.formula}) and
(\ref{formula.m.n.b.l}) imply that $m^{*}(h)^{2}+n^{*}(h)^{2}\to\infty$, as $h\to
0$, and thus
\[
\lim_{h\to 0}\frac{Hr_{1}(m(h),n(h))}{m(h)^{2}(1-\nu^{2})+H(m(h)^{2}+n(h)^{2})^{4}}
=\lim_{h\to 0}\frac{Hr_{1}(m^{*}(h),n^{*}(h))}{m^{*}(h)^{2}(1-\nu^{2})+H(m^{*}(h)^{2}+n^{*}(h)^{2})^{4}}
=0.
\]
Therefore,
\[
\frac{\tilde\Gl_{3}(h;m(h),n(h))}{\Gl^{*}_{3}(h;m(h),n(h))}\le\frac{\tilde\Gl_{3}(h)}{\Gl^{*}_{3}(h)}\le
\frac{\tilde\Gl_{3}(h;m^{*}(h),n^{*}(h))}{\Gl^{*}_{3}(h;m^{*}(h),n^{*}(h))},
\]
and (\ref{lim3}) follows.
\end{proof}
We have now achieved our goal, since (\ref{goal}) follows from (\ref{l32l3pr}) and
Lemma~\ref{lem:final}.

\subsection{Buckling modes}
\label{sec:altbc}
In this section we return to the original \bc s and the space $V_{h}$, defined
in (\ref{Breather}). Let
\begin{equation}
\label{lambda.Vh}
\Gl_{1}(h)=\inf_{\BGf\in\CA_{h}}\mathfrak{R}_1(h,\BGf).
\end{equation}
Even though, technically speaking, $V_{h}$ is not a subspace of $\tilde
V_{h}$, it is helpful to think of it as such. Hence, our next lemma is natural
(but not entirely obvious). 
\begin{lemma}
\label{lem:lamVh>lam.hat1}
Let $\Gl_{1}(h)$ and $\tilde\Gl_{1}(h)$ be given by (\ref{lambda.Vh}) and
(\ref{hatmn}), respectively, then
\begin{equation}
\label{lambda.Vh>lambda.hat1}
\Gl_1(h)\geq \tilde\Gl_1(h).
\end{equation}
\end{lemma}
\begin{proof}
In view of Theorem~\ref{th:mn} it is sufficient to prove the inequality
\[
\Gl_1(h)\geq\inf_{\myatop{m\ge 1}{n\ge 0}}\tilde\Gl_{m,n}(h).
\]
This is done by repeating the arguments in the proof of the analogous
inequality in Theorem~\ref{th:mn}. The argument is based on the fact the
$2L$-periodic extension of $\BGf\in\CA_h\subset V_{h}$, such that $\phi_{r}$
and $\phi_{\Gth}$ are even and $\phi_{z}$ is odd, is still of class $H^{1}$,
and the expansion (\ref{Sexp}) is valid. The inequality
(\ref{lambda.Vh>lambda.hat1}) follows from (\ref{mninf}) and the inequality
(\ref{mode.by.mode}), which is valid for each single Fourier mode.
\end{proof}
In order to prove that the asymptotic formula (\ref{answer}) holds
for $\Gl_{1}(h)$ (and hence for $\Gl_{\rm crit}(h)$) it is sufficient
to find a test function $\BGf^{h}\in \CA_{h}$ such that
\begin{equation}
  \label{Vhcond}
\lim_{h\to 0}\frac{ \mathfrak{R}_{1}(h,\BGf^{h})}{\Tld{\Gl}_{1}(h)}=1.  
\end{equation}
Indeed,
\[
1=\lim_{h\to 0}\frac{\mathfrak{R}_{1}(h,\BGf^{h})}{\Tld{\Gl}_{1}(h)}\ge
\lims_{h\to 0}\frac{\Gl_{1}(h)}{\Tld{\Gl}_{1}(h)}\ge\limi_{h\to 0}\frac{\Gl_{1}(h)}{\Tld{\Gl}_{1}(h)}\ge 1,
\]
which proves both that
\[
\lim_{h\to 0}\frac{\Gl_{1}(h)}{\Gl^{*}(h)}=1,
\]
and that $\BGf^{h}\in \CA_{h}$ is a buckling mode.

We construct the buckling mode $\BGf^{h}\in V_{h}$ as a 2-term Fourier
expansion (\ref{Fourier}). For this purpose we choose
$m=m(h)\to\infty$, as $h\to 0$, and $n=n(h)$ to lie on Koiter's circle and
\begin{equation}
  \label{bmFS}
  \begin{cases}
\displaystyle\phi^{h}_{r}(r,\Gth,z)=\sum_{m\in\{m(h),m(h)+2\}}f_{r}(r,m,n(h))\cos(n(h)\Gth)
\cos(\hat{m}z),\\[2ex]
\displaystyle\phi^{h}_{\Gth}(r,\Gth,z)=\sum_{m\in\{m(h),m(h)+2\}}f_{\Gth}(r,m,n(h))\sin(n(h)\Gth)
\cos(\hat{m}z),\\[2ex]
\displaystyle\phi^{h}_{z}(r,\Gth,z)=\sum_{m\in\{m(h),m(h)+2\}}f_{z}(r,m,n(h))\cos(n(h)\Gth)
\sin(\hat{m}z),
\end{cases}
\end{equation}
where now, in order to avoid confusion, we distinguish between $m\in\bb{Z}$ and 
\[
\hat{m}=\frac{\pi m}{L}.
\]
To ensure that $\BGf^{h}\in V_{h}$  we require
\[
f_{\Gth}(r,m(h)+2,n(h))=-f_{\Gth}(r,m(h),n(h)).
\]
The structure of coefficients
$\Bf(r,m,n)$ is determined by optimality at each of the two values of $m$
separately, since the expansion (\ref{Sexp}) is valid for $\BGf\in V_{h}$. In
particular, we 
choose
\[
f_{\Gth}(r,m(h),n(h))=ra_{\Gth}(h)+(r-1)n(h),\qquad
a_{\Gth}(h)=-\left.\frac{n(n^{2}+(\nu+2)\hat{m}^{2})}{(\hat{m}^{2}+n^{2})^{2}}\right|_{\myatop{m=m(h)}{n=n(h)}}.
\]
Let
\[
F_{z}(r,m,n,h)=a_{z}(m,n,h)+(r-1)\hat{m},\qquad 
a_{z}(m,n,h)=-\frac{\hat{m}(2\nu+(\nu+1)na_{\Gth}(h))}{2\hat{m}^{2}+(1-\nu)n^{2}},
\]
\[
F_{r}(r,m,n,h)=1-\frac{\nu(r-1)}{1-\nu}(na_{\Gth}(h)+1+\hat{m}a_{z}(m,n,h))-
\frac{\nu(r-1)^{2}}{2(1-\nu)}(na_{\Gth}(h)+n^{2}+\hat{m}^{2}).
\]
Then
\[
f_{r}(r,m(h),n(h))=F_{r}(r,m(h),n(h),h),\qquad
f_{r}(r,m(h)+2,n(h))=-F_{r}(r,m(h)+2,n(h),h),
\]
\[
f_{z}(r,m(h),n(h))=F_{z}(r,m(h),n(h),h),\qquad
f_{z}(r,m(h)+2,n(h))=-F_{z}(r,m(h)+2,n(h),h).
\]
\begin{figure}[t]
 \centering
\includegraphics[scale=0.5]{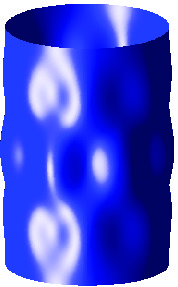}\hspace{8ex}
 \includegraphics[scale=0.5]{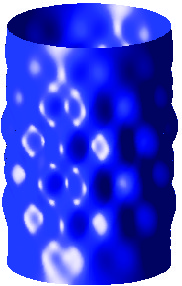}\hspace{8ex}
\includegraphics[scale=0.5]{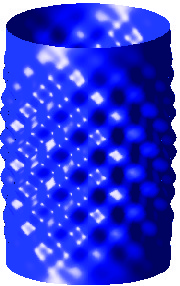}
 \caption{Koiter circle buckling modes corresponding, left to right, to
   $m(h)\sim h^{-1/8}$, $h^{-1/4}$ and $h^{-3/8}$. Poisson's ratio $\nu=1/3$.}
 \label{fig:fixedbc}
\end{figure}
Maple calculation verifies that the test function, $\BGf^{h}$ satisfies (\ref{Vhcond}).
Figure~\ref{fig:fixedbc} shows buckling modes for
\[
\hat{m}(h)=\left(\sqrt{\frac{2}{\Gl^{*}(h)}}\right)^{\Ga},\quad\Ga=1/4,\ 1,2,\ 3/4.
\]

\medskip

\noindent\textbf{Acknowledgments.}  The authors are grateful to Eric Clement
and Mark Peletier for their valuable comments and suggestions. This material
is based upon work supported by the National Science Foundation under Grant
No. 1008092.


\end{document}